\documentclass[a4paper]{scrartcl}

\usepackage{amsthm,amsmath,natbib}
\usepackage{amssymb}
\usepackage{latexsym}
\usepackage{enumerate}
\usepackage{graphicx,color}
\usepackage{hyperref}

\renewcommand{\d}{\mathrm{d}}
\newcommand{\E}{\mathrm{E}}
\numberwithin{equation}{section}

\numberwithin{figure}{section}

\theoremstyle{plain}
\newtheorem{theorem}{Theorem}[section]
\newtheorem{proposition}[theorem]{Proposition}

\newtheorem{lemma}[theorem]{Lemma}

\theoremstyle{definition}
\newtheorem{definition}[theorem]{Definition}

\newtheorem{example}[theorem]{Example}
\newtheorem{remark}[theorem]{Remark}

\title{
A martingale concept for non-monotone
information in a jump process framework
}

\author{Marcus C.~Christiansen\footnote{Carl von Ossietzky Universit\"{a}t Oldenburg, Institut f\"{u}r Mathematik, 26111 Oldenburg, Germany, marcus.christiansen@uni-oldenburg.de}}

\date{\today
}

\begin{document}

\maketitle

\begin{abstract}
The information dynamics in finance and insurance applications is usually modeled by a filtration. This paper looks at situations where information restrictions apply  such that the information dynamics may become non-monotone. A fundamental tool for calculating and managing risks in finance and insurance are martingale representations. We present a general theory that extends classical martingale representations to non-monotone information generated by  marked  point processes. The central idea is to focus only on those properties that martingales and compensators show on infinitesimally short intervals. While classical martingale representations describe innovations only,  our representations have an additional symmetric counterpart that quantifies the effect of information loss. We exemplify the results with examples from life insurance and credit risk.
\end{abstract}




\section{Introduction}

The value at time $t \in [0, T]$ of a financial claim $\xi \in L^1(\Omega , \mathcal{A},P)$ at time $T \in (0, \infty)$ is commonly calculated by
\begin{align}\label{EvaluationFormula}
  B(t)  \,\E_Q\Big[  \frac{\xi}{B(T)} \Big| \mathcal{F}_t \Big],
\end{align}
where $B$ is the value process of a risk-free asset,  $(\mathcal{F}_t)_{t \geq 0}$ is a filtration that describes the available information at each time $t\geq 0$, and $Q$ is some equivalent measure. For studying the time dynamics of the value process, we can exploit the fact that  $t \mapsto \E_Q[  \xi / B(T) | \mathcal{F}_t ]$ is always a martingale.

In this paper we suppose that information restrictions apply and replace the filtration $(\mathcal{F}_t)_{t \geq 0}$ by a family of sub-sigma-algebras $(\mathcal{G}_t)_{t \geq 0}$ that may be non-monotone, i.e.~we do not assume that $(\mathcal{G}_t)_{t \geq 0}$ is a filtration. We focus on modelling frameworks where  $(\mathcal{G}_t)_{t \geq 0}$ is generated by a marked point process, because this allows us to calculate martingale representations explicitly. Our approach seems to work also in more general settings, but a general theory is left to future research.

Information restrictions can be motivated by legal restrictions, data privacy efforts, information summarization or model simplifications.
An example for a legal information restriction is the General Data Protection Regulation 2016/679 of the European Union, which includes in Article 17 a so-called 'right to erasure', causing possible information loss.
\begin{example}[Evaluation of life insurance based on big data]\label{IntExpLifeIns}
Data from activity trackers, social media, etc.~can improve individual forecasts of the mortality and morbidity of insured persons. By exercising the 'right to erasure' according to the General Data Protection Regulation of the European Union, the policyholder may ask the insurer to delete parts of the health related data at discretion. Moreover, data providers might implement self-imposed  information restrictions for data privacy reasons. For example, users of Google products can opt for an  auto-delete of location history and activity data after a fixed time limit. As a result, the evaluation of an insurance liability $\zeta$ according to formula \eqref{EvaluationFormula} will be restricted to sub-sigma-algebras $(\mathcal{G}_t)_{t \geq 0}$ that  are non-monotone in $t$ due to data deletions.
\end{example}
Examples of information summarization can be found in Norberg (1991), where summarized life insurance values (retrospective and prospective reserves) are defined that encompass non-monotone information.
A popular model simplification is Markovian modeling even when the empirical data is not fully supporting the Markov assumption.
\begin{example}[Markovian approximations in credit rating models]\label{IntExpCredRat}
In the Jarrow-Lando-Turnbull model, the filtration $(\mathcal{F}_t)_{t \geq 0}$ is generated by a finite state space Markov chain $(R_t)_{t \geq 0}$ that represents credit ratings, cf.~Jarrow et al.~(1997). The Markov property makes it possible to equivalently replace $\mathcal{F}_t$ in \eqref{EvaluationFormula} by the sub-sigma-algebra $\mathcal{G}_t:=\sigma(R_t)$. The Markov assumption can be motivated by the theoretical idea that a credit rating should fully describe the current risk profile of a prospective debtor such that historical ratings can be ignored. However, empirical data does not always support the Markov property,  so that $\E_Q[  \xi / B(T) | \mathcal{G}_t ]$ may in fact differ from $\E_Q[  \xi / B(T) | \mathcal{F}_t ]$, cf.~Lando and Skodeberg (2002). 
The information dynamics of $\mathcal{G}_t=\sigma(R_t)$ is non-monotone in $t$.
\end{example}
Non-monotone information structures can also be found in Pardoux \& Peng (1994) and Tang \& Wu (2013), but in these papers specific independence assumptions make it possible to go back to filtrations and to work with classical martingale representations.

%
From now on we skip the subscript $Q$ in \eqref{EvaluationFormula} and all related expectations. Depending on the application, we  interpret $P$ either as the real world measure or as a risk-neutral measure.

When we replace filtration $(\mathcal{F}_t)_{t \geq 0}$ in \eqref{EvaluationFormula} by non-monotone information $(\mathcal{G}_t)_{t \geq 0}$, then all the powerful tools from martingale theory for studying the time dynamics of \eqref{EvaluationFormula} are not available anymore. In order to fill that gap, this paper derives general representations of the form
\begin{align}\label{MartingRepres}
\begin{split}
 \E[ \xi | \mathcal{G}_t]-\E[ \xi | \mathcal{G}_0]
 &= \sum_{I \in \mathcal{M}} \int_{(0,t ]\times E_I } G_I(u-,u,e) \,(\mu_I-\nu_I) (\d (u,e)) \\
 &\quad + \sum_{I \in \mathcal{M}}\int_{(0,t]\times E_I} G_I(u,u,e) \,(\rho_I- \mu_I) (\d (u,e)), \quad  t \geq 0,
\end{split}\end{align}
where $\xi$ is any integrable random variable,  $(\mathcal{G}_t)_{t \geq 0}$ is a non-monotone family of sigma-algebras generated by an extended marked point process that involves information deletions,
$(\mu_I)_{I \in \mathcal{M}}$ is a set of counting measures  that uniquely corresponds to the extended marked point process,
 $(\nu_I)_{I \in \mathcal{M}}$ and $(\rho_I)_{I \in \mathcal{M}}$ are infinitesimal forward and backward compensators of $(\mu_I)_{I \in \mathcal{M}}$, and the integrands  $G_I(u-,u,e)$ and $G_I(u,u,e)$ are adapted to the information at time $u-$ and time $u$, respectively.
In case that $(\mathcal{G}_t)_{t \geq 0}$ is increasing, i.e.~it is a filtration, the second line in \eqref{MartingRepres} is zero and the first line   conforms with classical martingale representations.  The central idea in this paper is to focus on those properties only that martingales and compensators show on infinitesimally small intervals. We call this the 'infinitesimal approach'. In principle, the infinitesimal approach is not restricted to point process process frameworks, but a fully general theory is beyond the scope of this paper. We will further extend our representation results to  processes of the form
\begin{align}\label{TypeOfProblem2}
  t \mapsto \E[ X_t | \mathcal{G}_t ], \quad  t \geq 0,
\end{align}
where $(X_t)_{t \geq 0}$ is a suitably integrable c\`{a}dl\`{a}g process. In this case an additional drift term appears on the right hand side of \eqref{MartingRepres}.

Martingale representations have various  applications in finance and insurance, and this in particular true for marked point process frameworks:
\begin{itemize}
\item  If a financial or insurance claim is hedgeable, then explicit hedges can be derived from  martingale representations, see e.g.~Norberg (2003) and Last \& Penrose (2011).
\item Martingale representations are a central tool for constructing and solving backward stochastic differential equations (BSDE), see e.g.~Cohen \& Elliott (2008), Bandini (2015)  and Confortola (2019). Many   optimal control problems in finance and insurance correspond to a BSDE problem, see e.g.~Cohen \& Elliott (2010) and Delong (2013).
\item Martingale representations can serve as additive risk factor decompositions, see Schilling et al.~(2020). An insurer needs to additively decompose the surplus from a policy or an insurance portfolio for regulatory reasons, see e.g.~M{\o}ller \& Steffensen (2007). Additive risk factor decompositions are also used in  finance, see e.g.~Rosen \& Saunders (2010).
\end{itemize}
In all three applications,  infinitesimal martingale representations according to \eqref{MartingRepres} allow us to include information restrictions into the modelling. We will study a hedging application for the model in Example \ref{IntExpCredRat}. We will see that estimation and calculation of hedging strategies under inappropriate Markov assumptions may unintentionally replace classical martingales by infinitesimal forward martingales (the first line on the right hand side of \eqref{MartingRepres}), and then the implied hedging error is just  the corresponding infinitesimal backward martingale part (the second line in \eqref{MartingRepres}). The application of infinitesimal martingale representations in BSDE theory is exemplarily discussed for Example \ref{IntExpLifeIns}. We will see that the integrands in \eqref{MartingRepres} correspond to the so-called sum at risk, which is a central figure in life insurance risk management. In Example \ref{IntExpLifeIns} we also briefly discuss risk factor decompositions. Information deletions upon request for  data privacy reasons can provoke arbitrage opportunities, and these can be split off as infinitesimal backward martingales, which is important for dealing with them.

Representation \eqref{MartingRepres} implies that $t \mapsto \E[ \xi | \mathcal{G}_t ]$ has a (unique) semimartingale modification. More generally, we will show that $t \mapsto \E[ X_t | \mathcal{G}_t ]$  has a (unique) semimartingale modification whenever $X$ is a semimartingale with integrable variation on compacts. 
The uniqueness and the semimartingale property are crucial in applications where the time dynamics shall be studied.  For example, in  life insurance the differential $\d E[ X_t | \mathcal{G}_t]$ might describe  the insurer's current surplus or loss at time $t$, cf.~Norberg (1992) and Norberg (1999).

The study of jump process martingales and their representations largely dates back to the 1970s, see e.g.~Jacod (1975), Boel et al.~(1975), Chou \& Meyer (1975), Davis (1976) and Elliott (1976). Since then extensions have been developed in different directions, see e.g.~Last \& Penrose (2011) and Cohen (2013). All of these papers stay within the framework of filtrations, i.e.~the information dynamics is monotone. The infinitesimal approach that we introduce here allows us to go beyond the framework of filtrations.
An elegant way to derive the classical martingale representation is a bare hands approach that starts with the Chou and Meyer construction of the martingale representation for a single jump process, followed by Elliott's extension to the case of ordered jumps. In this paper we also use a bare hands approach, but the classical stopping time concept is not applicable in our non-monotone information setting, so we need to leave the common paths.

The paper is organized as follows. In Section 2 we explain the basic concepts of the infinitesimal approach but avoid technicalities. In Section 3 we add technical assumptions and narrow the modelling framework down to pure jump process drivers. Section 4 verifies that \eqref{MartingRepres} is indeed a well-defined process. In Section 5 we identify infinitesimal compensators for a large class of jump processes. The central result \eqref{MartingRepres}  is  proven in Section 6 and extended to processes of the form  \eqref{TypeOfProblem2} in Section 7. In Section 8 we take a closer look at the Examples \ref{IntExpLifeIns} and \ref{IntExpCredRat}.

\section{The infinitesimal approach}

The central idea of the infinitesimal approach is to focus only on those properties that martingales and compensators show on infinitesimally short intervals. This section explains the basic ideas under the  general assumption that all limits in this section actually exist. Only from the next section on we narrow the framework down to pure jump process drivers, which is sufficient but not necessary to guarantee the existence of the limits. So, in general the infinitesimal approach is not restricted to jump process frameworks, but it is beyond the scope of this paper to find necessary conditions for the existence of the limits here.

Let $(\Omega, \mathcal{A}, P)$ be a complete probability space and let $\mathcal{Z}\subset \mathcal{A}$ be its null sets.  Let
 $\mathcal{F}=(\mathcal{F}_t)_{t \geq 0 }$ be a complete and right-continuous 
 filtration on this probability space.
We  interpret $\mathcal{F}_t$ as the observable information on the time interval $[0,t]$. 
Suppose that certain pieces of information expire after a finite holding time.  By subtracting from $\mathcal{F}_{t}$ all pieces of information that have expired until time $t$, we obtain the admissible information at time $t$.
We assume that this admissible information is represented by
a  
family of complete sigma-algebras $\mathcal{G}=(\mathcal{G}_t)_{t \geq 0 }$,
\begin{align*}
 \mathcal{G}_t \subseteq  \mathcal{F}_t, 
 \quad t \geq 0,
\end{align*}
which may be non-monotone in $t$.

A process $X$ is said to be adapted to the filtration $\mathcal{F}$  if $X_t$ is $\mathcal{F}_t$-measurable for each $t \geq 0$.
Likewise we say that a process $X$ is adapted to the possibly non-monotone information $\mathcal{G}$ if $X_t$ is $\mathcal{G}_t$-measurable for each $t \geq 0$. In addition to this classical concept,  we also take an incremental  perspective. 
\begin{definition}[incrementally adapted]\label{DefIncremAdap}
We say that  $X$ is incrementally adapted to $\mathcal{G}$ if $X_t-X_{s}$  is $\sigma( \mathcal{G}_{u}, u  \in [s,t])$-measurable for any interval $[s,t] \subset [0,\infty)$.
\end{definition}
In finance and insurance applications we think of $X$ as an aggregated  cash flow where the aggregated payments $X_t-X_s$ on the interval $[s,t]$ should depend only on the admissible information on $[s,t]$.
If $\mathcal{G}$ is a filtration, then incremental adaptedness is equivalent to classical adaptedness, but the two concepts differ for non-monotone information.

An integrable process $X$ is said to be a martingale with respect to $\mathcal{F}$ if it is $\mathcal{F}$-adapted and
\begin{align*}
   \E[ X_t - X_s | \mathcal{F}_s ] =0
\end{align*}
almost surely for each $0 \leq s \leq t$. Focussing on infinitesimally short intervals, in particular we have
\begin{align}\label{ClassMartInfProp}
    \lim_{n \rightarrow \infty} \sum_{\mathcal{T}_n} \E[ X_{t_{k+1}} - X_{t_{k}} | \mathcal{F}_{t_{k}}] =0
\end{align}
almost surely for each $t \geq 0$,  where $(\mathcal{T}_n)_{n\in \mathbb{N}}$ is any increasing sequence (i.e.~$\mathcal{T}_n \subset \mathcal{T}_{n+1}$ for all $n$) of partitions $0=t_0 < \cdots < t_n=t$ of the interval $[0,t]$ such that $|\mathcal{T}_n| := \max\{ t_{k}-t_{k-1}: k=1, \ldots, n\} \rightarrow 0$ for $n \rightarrow \infty$. In the literature we can find for \eqref{ClassMartInfProp} the intuitive notation $\E[ \d X_t | \mathcal{F}_{t-} ] =  0$.
\begin{definition}[infinitesimal martingales] \label{DefInfMarting}
Let $X$ be  incrementally adapted to $\mathcal{G}$. We say that $X$ is an infinitesimal forward/backward martingale (IF/IB-martingale) with respect to $\mathcal{G}$  if for each $t \geq 0$ and any increasing sequence of partitions $(\mathcal{T}_n)_{n\in \mathbb{N}}$ of $[0,t]$ with $\lim_{n \rightarrow \infty} |\mathcal{T}_n|=0$ we almost surely have
\begin{align*}
   \lim_{n \rightarrow \infty} \sum_{\mathcal{T}_n} \E[ X_{t_{k+1}} - X_{t_{k}} | \mathcal{G}_{t_{k}}] = 0
\end{align*}
 and
\begin{align*}
  \lim_{n \rightarrow \infty} \sum_{\mathcal{T}_n} \E[ X_{t_{k+1}} - X_{t_{k}} | \mathcal{G}_{t_{k+1}}] =0,
\end{align*}
respectively, given that the expectations and limits exist.
\end{definition}
%
%
Suppose now that $X$ is an $\mathcal{F}$-adapted and integrable counting process. The so-called compensator $C$ of $X$ is the unique $\mathcal{F}$-predictable finite variation process starting from $C_0=0$ such that $X-C$ is an $\mathcal{F}$-martingale. In particular, $C$ satisfies the equation
\begin{align}\label{InfCompensatorProperty}
  C_t  = \lim_{n \rightarrow \infty} \sum_{\mathcal{T}_n} \E[ X_{t_{k+1}} - X_{t_{k}} | \mathcal{F}_{t_{k}}]
\end{align}
almost surely for each $t \geq 0$, see Karr (1986, Theorem 2.17).
The intuitive notation for \eqref{InfCompensatorProperty}  is  $ \E[ \d X_t  | \mathcal{F}_{t-} ] = \d C_t$.
Furthermore, one can show that the $\mathcal{F}$-predictability of $C$ implies that
\begin{align}\label{InfCompensatorProperty2}
    \lim_{n \rightarrow \infty} \sum_{\mathcal{T}_n} \E[ C_{t_{k+1}} - C_{t_{k}} | \mathcal{F}_{t_{k}}]=C_t -C_0,
\end{align}
intuitively written as  $\E[ \d C | \mathcal{F}_{t-} ] =  \d C_t$. 
The latter fact motivates the following definition.
\begin{definition}[infinitesimally predictable processes]\label{DefIFIBpred}
We say that $X$ is infinitesimally forward/backward predictable (IF/IB-predictable) with respect to  $\mathcal{G}$  if for each $t \geq 0$ and any increasing sequence of partitions $(\mathcal{T}_n)_{n\in \mathbb{N}}$ of $[0,t]$ with $\lim_{n \rightarrow \infty} |\mathcal{T}_n|=0$ we almost surely have
\begin{align*}
   \lim_{n \rightarrow \infty} \sum_{\mathcal{T}_n} \E[ X_{t_{k+1}} - X_{t_{k}} | \mathcal{G}_{t_{k}}] = X_t-X_0
\end{align*}
and
\begin{align*}
  \lim_{n \rightarrow \infty} \sum_{\mathcal{T}_n} \E[ X_{t_{k+1}} - X_{t_{k}} | \mathcal{G}_{t_{k+1}}] =X_t-X_0,
\end{align*}
  respectively, given that the expectations and limits exist.
\end{definition}
Note that any IF/IB-predictable process is also incrementally adapted.
By combining \eqref{InfCompensatorProperty} and \eqref{InfCompensatorProperty2}, we obtain
\begin{align}\label{InfCompensatorProperty3}
 \lim_{n \rightarrow \infty} \sum_{\mathcal{T}_n} \E[ (X_{t_{k+1}}- C_{t_{k+1}}) - (X_{t_k}-C_{t_{k}}) | \mathcal{F}_{t_{k}}] =0
\end{align}
almost surely for each $t \geq 0$, which means that the process $X-C$ is an IF-martingale with respect to $\mathcal{F}$ according to Definition \ref{DefInfMarting}.
\begin{definition}[infinitesimal compensators]\label{DefInfCompens}
We say that a process $C$  is an infinitesimal forward/backward compensator of $X$ (IF/IB-compensator) with respect to $\mathcal{G}$ if $C$ is IF/IB-predictable and $X-C$ is an IF/IB-martingale with respect to $\mathcal{G}$, respectively.
\end{definition}
%
%
Let $\mathcal{G}_{[t_k,t_{k+1}]}:=\sigma( \mathcal{G}_{u}, u  \in [t_k,t_{k+1}])$ for any $t_{k+1} \geq t_k \geq 0$ and $\xi \in L^1(\Omega , \mathcal{A},P)$.
Then the construction
\begin{align*}
   \E[ \xi | \mathcal{G}_t]-\E[ \xi | \mathcal{G}_0]
   &=  \lim_{n \rightarrow \infty} \sum_{\mathcal{T}_n} \Big(\E[ \xi | \mathcal{G}_{t_{k+1}}]  -\E[ \xi | \mathcal{G}_{t_{k}}]\Big)\\
   & = \lim_{n \rightarrow \infty} \sum_{\mathcal{T}_n} \Big(\E[ \xi | \mathcal{G}_{[t_k,t_k+1]}]  -\E[ \xi | \mathcal{G}_{t_{k}}]\Big)\\
   & \quad - \lim_{n \rightarrow \infty} \sum_{\mathcal{T}_n} \Big(\E[ \xi | \mathcal{G}_{[t_k,t_k+1]}]  -\E[ \xi | \mathcal{G}_{t_{k+1}}]\Big)
\end{align*}
decomposes  the process $t\mapsto \E[ \xi | \mathcal{G}_t]$ into the difference of an IF-martingale and an IB-martingale, since
\begin{align*}
&\E\big[ \E[ \xi | \mathcal{G}_{[t_k,t_k+1]}]  -\E[ \xi | \mathcal{G}_{t_{k}}]\big| \mathcal{G}_{t_{k}}\big] =0,\\
&\E\big[ \E[ \xi | \mathcal{G}_{[t_k,t_k+1]}]  -\E[ \xi | \mathcal{G}_{t_{k+1}}]\big| \mathcal{G}_{t_{k+1}}\big] =0,
\end{align*}
 and since $\E[ \xi | \mathcal{G}_{[t_k,t_k+1]}]  -\E[ \xi | \mathcal{G}_{t_{k}}]$ and $\E[ \xi | \mathcal{G}_{[t_k,t_k+1]}]  -\E[ \xi | \mathcal{G}_{t_{k+1}}]$ are $\mathcal{G}_{[t_k,t_{k+1}]}$-measurable.
\begin{definition}[infinitesimal martingale representation]
We say that $\E[ \xi | \mathcal{G}_t]-\E[ \xi | \mathcal{G}_0] = F_t - B_t$, $t \geq 0$, is an infinitesimal martingale representation if $F$ is an IF-martingale and $B$ is an IB-martingale with respect to $\mathcal{G}$.
\end{definition}
Suppose now that $X$ describes a discounted claim process in a finance or insurance application. Then we  are typically interested in the process $t \mapsto \E [ X_t | \mathcal{F}_t]$, which is not necessarily well-defined. If $X$ is a  c\`{a}dl\`{a}g process whose suprema on compacts have finite expectations, then  there exists a unique c\`{a}dl\`{a}g process $X^{\mathcal{F}}$, the so-called optional projection of $X$ with respect to $\mathcal{F}$, such that
\begin{align*}
  X^{\mathcal{F}}_t = \E [ X_t | \mathcal{F}_t]
\end{align*}
almost surely for each $t \geq 0$.  We say here that a process is unique if it is unique up to evanescence. We now expand the concept of optional projections to non-monotone information.
\begin{definition}[optional projection]\label{DefOptProj} Let $X$ be an integrable c\`{a}dl\`{a}g process.
If there exists a unique c\`{a}dl\`{a}g process $X^{\mathcal{G}}$ such that
\begin{align*}
  X^{\mathcal{G}}_t = \E [ X_t | \mathcal{G}_t]
\end{align*}
almost surely for each $t \geq 0$, then we call $X^{\mathcal{G}}$ the optional projection of $X$ with respect to  $\mathcal{G}$.
\end{definition}
The optional projection $X^{\mathcal{G}}$ can be decomposed to
\begin{align*}
   \E[ X_t | \mathcal{G}_t]-\E[ X_0 | \mathcal{G}_0]
   &=  \lim_{n \rightarrow \infty} \sum_{\mathcal{T}_n} \Big(\E[ X_{t_{k+1}} | \mathcal{G}_{t_{k+1}}]  -\E[ X_{t_k} | \mathcal{G}_{t_{k}}]\Big)\\
   & = \lim_{n \rightarrow \infty} \sum_{\mathcal{T}_n} \Big(\E[  X_{t_{k}} | \mathcal{G}_{[t_k,t_k+1]}]  -\E[  X_{t_{k}} | \mathcal{G}_{t_{k}}]\Big)\\
   & \quad - \lim_{n \rightarrow \infty} \sum_{\mathcal{T}_n} \Big(\E[  X_{t_{k}} | \mathcal{G}_{[t_k,t_k+1]}]  -\E[  X_{t_{k}} | \mathcal{G}_{t_{k+1}}]\Big)\\
   & \quad  + \lim_{n \rightarrow \infty} \sum_{\mathcal{T}_n}   \E[ X_{t_{k+1}}- X_{t_{k}} | \mathcal{G}_{t_{k+1}}],
\end{align*}
which is a sum of an IF-martingale, an IB-martingale and an IB-compensator with respect to $\mathcal{G}$. By switching the roles of $t_k$ and $t_{k+1}$ we can obtain a similar decomposition where the IB-compensator is replaced by an IF-compensator.
\begin{definition}[infinitesimal representation of optional projections]
We call $\E[ X_t | \mathcal{G}_t]-\E[ \xi | \mathcal{G}_0] = F_t - B_t + C_t$,  $t \geq 0$,  an infinitesimal representation if $F$ is an IF-martingale, $B$ is an IB-martingale, and $C$ is either an IB-compensator or an IF-compensator with respect to $\mathcal{G}$.
\end{definition}
As we mentioned at the beginning of this section, so far we simply assumed that all the limits that we discussed here indeed exist. In the next section we focus on a marked point process framework, since this guarantees not only the existence of the limits but also allows us to calculate the limits explicitly.

\section{Jump process framework}\label{SectionFramework}

In the literature, we can find different approaches for defining a jump process framework. One way is to start with a marked point process $(\tau_i,\zeta_i)_{i \in \mathbb{N}}$  on $(\Omega, \mathcal{A}, P)$
with some measurable mark space $(E, \mathcal{E})$, i.e.
\begin{itemize}
  \item the $\tau_i:(\Omega, \mathcal{A}) \rightarrow ([0,\infty], \mathcal{B}([0,\infty]))$, $i \in \mathbb{N}$ are random times,
  \item the $\zeta_i: (\Omega, \mathcal{A}) \rightarrow   (E, \mathcal{E})$ are random variables giving the marks.
\end{itemize}
Different from the point process literature,  we  do not assume here that the random times $(\tau_i)_{i\in \mathbb{N}}$ are increasing or ordered in any specific way. This gives us useful modelling flexibility, see also the comments at the end of this section.
Let $E$ be a separable complete metric space and $\mathcal{E}:=\mathcal{B}(E)$ its Borel sigma-algebra. Moreover, let $\Omega$ be a Polish space and $\mathcal{A}$ its Borel sigma algebra. 
We interpret each $\zeta_i$ as a piece of information  that can be observed from time $\tau_i$ on.
As motivated in the introduction, we additionally assume that the information pieces $\zeta_i$  are possibly deleted after a finite holding time.
Therefore, we expand the marked point process $(\tau_i,\zeta_i)_{i \in \mathbb{N}}$ to $(\tau_i, \zeta_i, \sigma_i )_{i \in \mathbb{N}}$, where
\begin{itemize}
  \item the $\sigma_i:(\Omega, \mathcal{A}) \rightarrow ([0,\infty], \mathcal{B}([0,\infty]))$, $i \in \mathbb{N}$, are random times such that $\tau_i \leq  \sigma_i $.
\end{itemize}
We interpret $\sigma_i$ as the deletion time of information piece $\zeta_i$.  Note that the
random times $(\sigma_i)_{i \in \mathbb{N}}$ are in general not ordered. For the sake of a more compact notation, in the following we will work with the equivalent sequence $(T_i,Z_i)_{i \in \mathbb{N}}$ defined as
\begin{align*}
  T_{2i-1} := \tau_i, \quad T_{2i} := \sigma_i, \quad Z_{2i-1}:= \zeta_i, \quad Z_{2i}:= \zeta_i, \quad i \in \mathbb{N},
\end{align*}
i.e.~the random times $T_{2i-1}$ with odd indices refer to innovations and the consecutive random times $T_{2i}$ with even indices are  the corresponding deletion times.
We generally assume that
\begin{align}\label{FinExpJumpNo}
  \E\bigg[ \sum_{i =1}^{\infty} \mathbf{1}_{\{T_i \leq t\}} \bigg] < \infty, \quad t \geq 0,
\end{align}
which will ensure the existence of (infinitesimal) compensators. Condition \eqref{FinExpJumpNo} implies that almost surely there are at most finitely many random times on bounded intervals.
Moreover, we  assume  that
\begin{align}\begin{split}\label{NoMultiJumpsAssumption}
&T_{2i-1}(\omega) < T_{2i}(\omega), \quad \omega \in \{T_{2i} < \infty\}, \, i \in \mathbb{N},
\end{split}\end{align}
i.e.~a new piece of information is not instantaneously deleted but is available for at least a short amount of time.
Based on the sequence $(T_i,Z_i)_{i \in \mathbb{N}}$ we generate random counting measures
$\mu_I$  via
\begin{align*}
  \mu_I([0,t]\times  B  ) &:=  \mathbf{1}_{\{t \geq T_i=T_j : i, j \in I\}\cap \{T_i\neq T_j : i\in I, j \not\in I\}} \mathbf{1}_{\{Z_I \in B\}}
\end{align*}
for $t \geq 0$, $ B \in \mathcal{E}_I$, $ I\subseteq \mathbb{N}$, where
\begin{align*}
  \mathcal{E}_I:=\mathcal{B}(E_I), \quad E_I:={E}^{|I|}, \quad Z_I:=(Z_i)_{i \in I}.
\end{align*}
If the different random times $(T_i)_i$ never coincide, then we just need to consider the counting measures $\mu_{\{i\}}$, $i \in \mathbb{N}$, which describe separate arrivals of  the random times $T_i$ and their marks $Z_i$.  But if random times can occur simultaneously, then we need the full scale of  counting measures $\mu_{I}$, $ I \subseteq \mathbb{N}$, which cover all kinds of separate and joint events.
For each $ I \subseteq \mathbb{N}$, the measures $\{\mu_I(\cdot)(\omega)| \omega \in \Omega\}$  generated by their values on $[0,t] \times  B $  form a random counting measure on  $([0, \infty) \times E_I  , \mathcal{B}([0, \infty) \times  E_I ) )$, i.e.
\begin{itemize}
  \item for any fixed $A \in \mathcal{B}([0,\infty) \times E_I)$ the mapping $\omega \mapsto \mu_I ( A)(\omega)$ is measurable from $(\Omega, \mathcal{A})$ to $(\overline{\mathbb{N}}_0, \mathcal{B}(\overline{\mathbb{N}}_0))$ with $\overline{\mathbb{N}}_0:= \mathbb{N}_0 \cup \{\infty\}$,
  \item for almost each $\omega \in \Omega$ the mapping $A \mapsto \mu_I (A)(\omega)$ is a locally finite measure on $([0, \infty) \times E_I  , \mathcal{B}([0, \infty) \times E_I) )$.
\end{itemize}
The observable information at time $t \geq 0$  is given  by the complete 
filtration
\begin{align*}
 \mathcal{F}_t &:=  \sigma\Big( \{ T_{2i-1} \leq s < T_{2i}\} \cap \{ Z_{2i} \in B\} : s \in [0,t],  B\in \mathcal{E}, i \in \mathbb{N}\Big)   \vee \mathcal{Z},
\end{align*}
which lets the random times  $T_i$, $i \in \mathbb{N}$, be stopping times. Here the symbol '$\vee$' denotes the sigma algebra that is generated by the union of the involved sets.
 The admissible information at time $t\geq 0 $ is given by the  family of sub-sigma-algebras
\begin{align*}
 \mathcal{G}_t=\sigma\Big(  \{ T_{2i-1} \leq t < T_{2i}\} \cap \{ Z_{2i} \in B\}:  B\in \mathcal{E}, i \in \mathbb{N}\Big)   \vee \mathcal{Z}.
\end{align*}
The admissible information immediately before time $t> 0 $ is given by the  family of sub-sigma-algebras
\begin{align*}
 \mathcal{G}^-_{t} =\sigma\Big( \{ T_{2i-1} < t \leq T_{2i}\} \cap \{ Z_{2i} \in B\}:  B\in \mathcal{E}, i \in \mathbb{N}\Big)   \vee \mathcal{Z}.
\end{align*}
\begin{remark}\label{RemarkNonInformativeOrder}
Recall that $T_{2i-1} \leq T_{2i}$, $i \in\mathbb{N}$, is the only kind of order that we assume to hold between the random times $(T_i)_i$, resulting from the natural assumption $\tau_i \leq \sigma_i $,  $i \in\mathbb{N}$. This fact is relevant when an ordering unintentionally reveals additional information. For example, if we have a model where the innovation times $(\tau_i)_i$ are ordered, i.e.~$T_1 < T_3 < T_5 < \cdots $, then $\mathcal{G}_t$ reveals among other things the exact number of deletions that have happened until $t$.  This can be an unwanted feature if the number of past deletions is itself a non-admissible piece of information. In many situations we can avoid such an implied information effect by ordering the pairs $(T_{2i-1},T_{2i})$ in a non-informative way.
\end{remark}
\begin{remark}\label{RemarkInformRepres}
Without loss of generality suppose here that $0 \not\in E$. Then, by defining the c\`{a}dl\`{a}g process
\begin{align*}
  \Gamma_t := (Z_{2i} \mathbf{1}_{\{T_{2i-1} \leq t < T_{2i}\}})_{i \in \mathbb{N}},
\end{align*}
the information $\mathcal{G}_t$ and $\mathcal{G}^-_t$ can be alternatively  represented as
\begin{align*}
\mathcal{G}_t =& \sigma(\Gamma_t) \vee \mathcal{Z}, \quad t \geq0 ,\\
\mathcal{G}^-_t =& \sigma(\Gamma_{t-}) \vee \mathcal{Z}, \quad t >0.
\end{align*}
\end{remark}

\section{Optional projections}

In this section we study existence and path properties of optional projections. Note that this and all following sections generally assume that we are in the marked point process framework of section \ref{SectionFramework}.  Recall also our specific definition of $\mathcal{G}^-_t$.
\begin{theorem}\label{PropDefOptionalProjection}
Suppose that $X=(X_t)_{t\geq 0}$ is a c\`{a}dl\`{a}g process that satisfies
\begin{align}\label{IntegrabCond}
  \E\bigg[  \sup_{0 \leq s \leq t} | X_s| \bigg] < \infty, \quad t \geq 0.
\end{align}
Then the optional projection $X^{\mathcal{G}}$ according to Definition \ref{DefOptProj} exists, and
 we have $X^{\mathcal{G}}_{t-}=\E[ X_{t-} | \mathcal{G}^-_{t}]$  almost surely for each $t > 0$. If $X$ has integrable  variation on compacts, then $X^{\mathcal{G}}$ has paths of finite variation on compacts.
\end{theorem}%
It might be surprising here that $X^{\mathcal{G}}$ is indeed always a c\`{a}dl\`{a}g process, but note that condition \eqref{FinExpJumpNo} rules out clusters of jump times in our marked point process framework.
Before we turn to the proof of Theorem \ref{PropDefOptionalProjection}, we develop several auxiliary results.  Let
\begin{align*}
\mathcal{N}&:= \{ M \subset \mathbb{N} :  \# M < \infty \},\\
  \mathcal{M}&:= \{ M \subset \{1,3,5,\ldots\} :  \# M < \infty \}
\end{align*}
be  all finite subsets of the natural numbers and all finite subsets of the odd natural numbers,
and define
\begin{align*}
 R_I&:= (Q_I,(Z_i)_{i \in I}), \quad 
  I \in \mathcal{N},
\end{align*}
where $Q_I:=\sup\{t \geq 0 : \mu_I([0,t] \times E_I)=0\} $.

Since $\Omega$  is a Polish space and $\mathcal{A}$ its Borel sigma algebra, there exist regular conditional distributions $ P( \cdot | Z_M)$ and  $ P( \cdot | Z_M,R_I)$  on $(\Omega, \mathcal{A})$  for each $M \in \mathcal{M}$ and $I \in \mathcal{N}$. As the sets $\mathcal{M}$ and $\mathcal{N}$ are countable, all these conditional distributions are simultaneously unique up to a joint exception zero set.  In this paper the notation
\begin{align*}
  P_{M,R_I}(\cdot )  = P(\,\cdot\,| Z_M,R_I)
\end{align*}
refers to an arbitrary but fixed regular version of the conditional expectation on the right hand side,
and for any integrable random variable $Z$ we set
\begin{align*}
\E_{M,R_I}[Z] := \int Z\,\d P_{M,R_I}
\end{align*}
i.e.~$\E_{M,R_I}[Z ]$ is the specific version of the conditional expectation $\E[ Z | Z_M,R_I]$ that we obtain by integrating $Z$ with respect to the specific regular versions that we picked for  $P(\,\cdot\,| Z_M,R_I) $.
 In case of $I= \emptyset$ we also use the short forms  $P_M=P_{M,R_{\emptyset}}$ and $\E_M=\E_{M,R_{\emptyset}}$ since $P_{M,R_{\emptyset}}$ is a version of $P(\cdot| Z_M)$.

Moreover, with defining $I-1:= \{i-1: i \in I\}$, the mappings
\begin{align}\label{ConventionCondProb}\begin{split}
P_{M,R_I=r}(\cdot ) :=  P(\,\cdot \,| Z_{M_I}=z, R_I=r)|_{z=Z_{M_I}}, \quad M_I := M \setminus (I \cup (I-1))
\end{split}\end{align}
refer to arbitrary but fixed regular versions of the factorized conditional expectations on the right hand side,
and for any integrable random variable $Z$ we define
\begin{align*}
  \E_{M,R_I=r}[Z] &:= \int Z\, \d P_{M,R_I=r}.
\end{align*}
By reducing  $M$ down to $M_I$  we leave out exactly those random variables in $Z_M$ that are already covered by $R_I$.
Note that the mapping  $P_{M,R_I=r}(\cdot )|_{r=R_I}$  equals $P_{M,R_I} (\cdot)$.

For $M \in \mathcal{M}$ and $t \geq 0$ we define the $\mathcal{G}_t$-measurable sets
\begin{align*}
A^M_t&:= \bigcap_{i \in M}  \{ T_i \leq t < T_{i+1}\} \cap \bigcap_{i \not\in M} (\Omega \setminus \{ T_i \leq t < T_{i+1}\})
\end{align*}
and  corresponding $\mathcal{G}$-adapted stochastic processes $\mathbb{I}^M=(\mathbb{I}_{t}^{M})_{t\geq 0}$ via
\begin{align}\label{DefIndProc}\begin{split}
  \mathbb{I}_{t}^{M} &:= \mathbf{1}_{A^M_t}, \quad t \geq 0.
  \end{split}\end{align}
Because of assumption \eqref{FinExpJumpNo} the paths of $\mathbb{I}_{t}^{M}$ have finitely many jumps on compacts only, so they have left and right limits. Moreover, by construction they are right-continuous, so the processes $\mathbb{I}^{M}$ are   c\`{a}dl\`{a}g. The left limits can be represented as $\mathbb{I}_{t-}^{M} = \mathbf{1}_{A^M_{t-}}$
where
\begin{align*}
A^M_{t-}&:= \bigcap_{i \in M}  \{ T_i < t \leq T_{i+1}\} \cap \bigcap_{i \not\in M} (\Omega \setminus \{ T_i < t \leq T_{i+1}\}).
\end{align*}
\begin{proposition} \label{DarstellungBedErwartBzglG}
For any integrable random variable  $\xi$ and any sets $M \in \mathcal{M}$ and  $I \in \mathcal{N}$
 we almost surely have
\begin{align}\label{EquivOfCondExpectations} \begin{split}
\mathbb{I}^M_t  \E [ \xi | \mathcal{G}_t \vee \sigma(R_I) ]&=    \mathbb{I}^M_t  \frac{ \E_{M,R_I} [ \xi \mathbb{I}^M_t ]}{\E_{M,R_I} [ \mathbb{I}^M_t   ]} ,\\
\mathbb{I}^M_{t-}  \E [ \xi | \mathcal{G}^-_{t} \vee \sigma(R_I) ]&=    \mathbb{I}^M_{t-}  \frac{ \E_{M,R_I}[ \xi \mathbb{I}^M_{t-} ]}{\E_{M,R_I} [ \mathbb{I}^M_{t-} ]}
\end{split}\end{align}
under the convention that $0/0:=0$.
\end{proposition}
Note here that $\sigma(R_I)$ equals the trivial sigma-algebra if $I = \emptyset$.
Whenever  $\E_{M,R_I} [ \mathbb{I}^M_t]=0$ and  $\E_{M,R_I} [ \mathbb{I}^M_{t-} ]=0$, we necessarily have $\E_{M,R_I} [ \xi \mathbb{I}^M_t ] =0$ and $\E_{M,R_I} [ \xi \mathbb{I}^M_{t-}]=0$, respectively, so the right hand sides of \eqref{EquivOfCondExpectations} are indeed well defined.
\begin{proof}
The left hand side of \eqref{EquivOfCondExpectations} almost surely equals the conditional expectations that one obtains when the sigma-algebras $\mathcal{G}$ and $\mathcal{G}^-$ are replaced by their non-completed versions. Therefore, in the remaining proof we will ignore the extension by $\mathcal{Z}$ in the definitions of $\mathcal{G}$ and $\mathcal{G}^-$.

For each $H \in \sigma(Z_M) $ there exists a $G \in \mathcal{G}_t $ such that $ H \cap A_t^{M} = G \cap A_t^{M} $ and vice versa. Thus,
\begin{align}\label{SpurSigmaAlgebren1}
  (\sigma(Z_M) \vee \sigma(R_I)) \cap  A_t^{M} = (\mathcal{G}_t  \vee \sigma(R_I))\cap  A_t^{M} \subseteq  \mathcal{G}_t \vee  \sigma(R_I),  \quad t \geq 0.
\end{align}
This implies that the random variable $\mathbb{I}^M_t  \frac{ \\E_{M,R_I} [ \xi \mathbb{I}^M_t ]}{\E_{M,R_I}[ \mathbb{I}^M_t]}$ is $(\mathcal{G}_t\vee \sigma(R_I))$-measurable, and
for each $G \in \mathcal{G}_t \vee \sigma(R_I)$ we obtain
\begin{align*}
  \E \bigg[  \mathbf{1}_{G}  \mathbb{I}^M_t  \frac{ \E_{M,R_I} [ \xi \mathbb{I}^M_t ]}{\E_{M,R_I} [ \mathbb{I}^M_t  ]} \bigg] & =  \E \bigg[  \E_{M,R_I} \bigg[ \mathbf{1}_{H}  \mathbb{I}^M_t  \frac{ \E_{M,R_I} [ \xi \mathbb{I}^M_t ]}{\E_{M,R_I} [ \mathbb{I}^M_t ]}  \bigg] \bigg]\\
  & = \E \big[   \mathbf{1}_{H}  \E_{M,R_I} [ \xi \mathbb{I}^M_t ] \big]\\
  & = \E [   \mathbf{1}_{G}  \mathbb{I}^M_t \xi   ]\\
  & = \E \big[ \mathbf{1}_{G}  \mathbb{I}^M_t \E[   \xi | \mathcal{G}_t\vee \sigma(R_I) ]  \big],
\end{align*}
i.e.~the first equation in \eqref{EquivOfCondExpectations} holds. By replacing \eqref{SpurSigmaAlgebren1} by
\begin{align}\label{SpurSigmaAlgebren}
(\sigma(Z_M) \vee \sigma(R_I)) \cap  A_{t-}^{M} = (\mathcal{G}^-_t  \vee \sigma(R_I))\cap  A_{t-}^{M} \subseteq  \mathcal{G}^-_t \vee  \sigma(R_I),  \quad t \geq 0,
\end{align}
we can analogously show that the second equation in \eqref{EquivOfCondExpectations} holds.
\end{proof}
\begin{lemma}\label{LemmaLimits}
For each $M \in \mathcal{M}, I \in \mathcal{N}, r \geq 0, e \in E$ and each c\`{a}dl\`{a}g process $X$ that satisfies condition \eqref{IntegrabCond}, the stochastic processes
\begin{align*}
& t \mapsto \E_{M,R_I}[X_t \mathbb{I}_{t}^{M} ],\\
& t \mapsto \E_{M,R_I=r}[X_t \mathbb{I}_{t}^{M}   ]
 \end{align*}
have c\`{a}dl\`{a}g paths. Moreover, their left limits can be obtained by replacing $X_t \mathbb{I}_{t}^{M}$ by $X_{t-} \mathbb{I}_{t}^{M}$.
 \end{lemma}
\begin{proof}
Apply the Dominated Convergence Theorem.
\end{proof}
%
%

\begin{proposition}\label{LemmaBoundedFractions}
Under the conventions $0/0:=0$ and $1/0:=\infty$, for each $M\in \mathcal{M} $  we almost surely have
\begin{align*}
\sup_{t \in [0,\infty)} \frac{\mathbb{I}_{t}^{M}}{\E_M [ \mathbb{I}^M_t ]}  < \infty.
\end{align*}
\end{proposition}
\begin{proof}
Let $\tau$ and $\sigma$ be any two non-negative random times such that $\tau \leq \sigma$. At first we are going to show that
\begin{align}\label{BoundedQuotient1}
Z:= \sup_{t \in [0,\infty)} \frac{\mathbf{1}_{\{\tau \leq t < \sigma \}}}{\E [ \mathbf{1}_{\{\tau \leq t < \sigma \}} ]}  < \infty
\end{align}
almost surely. For $(t,s) \in [0,\infty)^2$ we define the unbounded rectangles $A_{(t,s)}:= \{ (t',s') : t'\leq t, s< s'\}$ and the countably generated set
\begin{align*}
  B&:= \bigcup_{(t,s) \in \beta} A_{(t,s)}, \quad \beta:=\Big\{(t,s) \in \mathbb{Q}_+^2 : t<s, P((\tau,\sigma) \in A_{(t,s)})=0\Big\} .
\end{align*}
 Let $\partial B$ and $B^{\circ}$ be the boundary and the interior of $B$. Any line of the form $L_x:= \{ (x,x)+ \lambda \, (1,-1): \lambda \in \mathbb{R}\}$ intersects $\partial B$ at most at one point, since for any two points $y,y' \in L_x$, $y\neq y'$ we either have $y \in A_{y'}^{\circ}$ or $y' \in A_{y}^{\circ}$. Therefore,  the set
\begin{align*}
\gamma:=\bigg(\bigcup_{x \in \mathbb{Q}_+} L_x\bigg)  \cap \partial B \cap \Big\{ (t,s) \in [0,\infty)^2 : P((\tau,\sigma) \in A_{(t,s)})=0\Big\}
\end{align*}
is countable, and
\begin{align*}
  C&:= \bigcup_{(t,s) \in \gamma} A_{(t,s)}
\end{align*}
is countably generated. The sets  $N_B=\{(\tau,\sigma) \in B\}$ and $N_C=\{(\tau,\sigma) \in C\}$  are both null sets since they equal countable unions of null sets.

Suppose now that $Z(\omega) = \infty$ for an arbitrary but fixed $\omega \in \Omega$. We necessarily have $\tau(\omega) < \sigma(\omega)$.
Since $t \mapsto \E [ \mathbf{1}_{\{\tau \leq t < \sigma \}} ]$ is a c\`{a}dl\`{a}g function,  at least one of the following statements is true:
\begin{enumerate}[(1)]
  \item $\E [ \mathbf{1}_{\{\tau \leq u < \sigma \}} ]=0$  for some $u \in (\tau(\omega),\sigma(\omega))$,
  \item $\E [ \mathbf{1}_{\{\tau < u \leq \sigma \}} ]=0$ for some $u \in (\tau(\omega),\sigma(\omega))$,
  \item $\E [ \mathbf{1}_{\{\tau \leq u < \sigma \}} ]=0$  for  $u = \tau(\omega)$,
  \item $\E [ \mathbf{1}_{\{\tau < u \leq \sigma \}} ]=0$  for $u = \sigma(\omega)$.
\end{enumerate}

In case (1)  we have $P((\tau,\sigma)\in A_{(u,u)}) =\E [ \mathbf{1}_{\{\tau \leq u < \sigma \}} ] =0$ and $(\tau(\omega),\sigma(\omega)) \in A_{(u,u)}^{\circ}$, such that can conclude that $\omega \in N_B$.

In case (2) we can argue analogously to case (1), but we need to replace the definition of $A_{(t,s)}$ by $ \{ (t',s') : t'< t, s\leq s'\}$ and define a corresponding null set  $N_B'$. We obtain that  $\omega \in N_B'$.

In case  (3) we have $P((\tau,\sigma)\in A_{(\tau(\omega),\tau(\omega))})=\E [ \mathbf{1}_{\{\tau \leq \tau(\omega) < \sigma \}} ] =0$ and $(\tau(\omega),\sigma(\omega)) \in A_{(\tau(\omega),\tau(\omega))} \subset B  \cup \partial B$. If $(\tau(\omega),\sigma(\omega)) \in B$, then $\omega \in N_B$. If $(\tau(\omega),\sigma(\omega)) \in \partial B $, then the whole line segment $ \{ \theta (\tau(\omega), \tau(\omega)) + (1-\theta) (\tau(\omega), \sigma(\omega)) : \theta  \in (0,1)\}$ is in $\partial B$, because of $(\tau(\omega), \tau(\omega)) \in \partial B$ and the rectangular shape of the sets $A_{(t,s)}$. On this line there is at least one intersection with  $C$, such that we can conclude that $\omega \in N_C$.

In case (4) we can argue similarly to case (3), but we need to replace the definition of $A_{(t,s)}$ by $ \{ (t',s') : t'< t, s\leq s'\}$ and define corresponding null sets  $N_B'$ and $N_C'$.

All in all, we have $P(Y=\infty) \leq P(N_B \cup N_C\cup N_B' \cup N_C')=0$, i.e.~equation \eqref{BoundedQuotient1} holds.

Now, let $M \in \mathcal{M}$ be arbitrary but fixed and choose $\tau$ and $\sigma$ as the random times where $\mathbb{I}_{t}^M$ jumps from zero to one and jumps back to zero, respectively. Suppose that  $P_{Z_M=z}$ is a regular version of $P( \cdot | Z_M = z)$ and  $\E_ {Z_M=z}[\cdot]$ its  corresponding expectation.  Then from \eqref{BoundedQuotient1} we can conclude that
\begin{align*}
P_{Z_M=z}\bigg( \sup_{t \in [0,\infty)} \frac{\mathbb{I}_{t}^{M}}{\E_{Z_M=z} [ \mathbb{I}_{t}^{M} ]}  = \infty \bigg) =0
\end{align*}
for each choice of $z$. Replacing both $z$ by $Z_M$, where we use the insertion rule for conditional expectations for the inner $z$,  and  taking the unconditional expectation on both hand sides of the equation, we end up with
\begin{align*}
P\bigg( \sup_{t \in [0,\infty)} \frac{\mathbb{I}_{t}^{M}}{\E_{Z_M} [ \mathbb{I}_{t}^{M} ]}  = \infty \bigg) =0.
\end{align*}
\end{proof}
\begin{proof}[Proof of Theorem \ref{PropDefOptionalProjection}]
Motivated by Proposition \ref{DarstellungBedErwartBzglG}, we set
\begin{align*}
Y_t:= \sum_{M \in  \mathcal{M}}  \mathbb{I}^M_t  \frac{ \E_M [ X_t \mathbb{I}^M_t ]}{\E_M [ \mathbb{I}^M_t ]}, \quad t \geq 0,
\end{align*}
since this process almost surely equals $X^{\mathcal{G}}_t$ for each $t \geq 0$.
Note that there are at most a countable number of conditional expectations involved, so the corresponding regular versions are simultaneously unique up to evanescence.
For each compact interval  $[0,t]$ and  almost each $\omega \in \Omega$,
the set
\begin{align}\label{DefOfMt}
\mathcal{M}_t(\omega):= \{ M \in \mathcal{M}: \mathbb{I}^M_{u}(\omega) =1 \textrm{ for at least one }u \in [0,t]\}
\end{align}
is finite  because of assumption \eqref{FinExpJumpNo}. In case of $\E_M [ \mathbb{I}^M_t  ](\omega)\neq 0$ Lemma \ref{LemmaLimits} yields that
\begin{align}\label{RechtsStetigkeitDomConv}\begin{split}
  \lim_{\varepsilon \downarrow 0} Y_{t+\varepsilon}(\omega)  =   & \sum_{M \in \mathcal{M}_{t+1}(\omega)}  \lim_{\varepsilon \downarrow 0}
   \mathbb{I}^M_{t+\varepsilon}(\omega)   \frac{ \E_M [ X_{t+\varepsilon} \mathbb{I}^M_{t+\varepsilon} ](\omega)}{\E_M [ \mathbb{I}^M_{t+\varepsilon} ](\omega)}   \\
     &=\sum_{M \in \mathcal{M}_{t+1}(\omega)}  \mathbb{I}^M_t(\omega)  \frac{ \E_M [ X_t \mathbb{I}^M_t ](\omega)}{\E_M [ \mathbb{I}^M_t  ](\omega)}\\
     & = Y_t(\omega).
\end{split} \end{align}
 In case of $\E_M [ \mathbb{I}^M_t  ](\omega)=0$, Proposition \ref{LemmaBoundedFractions} implies that
$\mathbb{I}^M_{t}=0$  for almost all $\omega \in \Omega$, where the exception zero set does not depend on the choice of $t$.  So \eqref{RechtsStetigkeitDomConv} is almost surely  true on $[0,\infty)$ since $\mathbb{I}^M_{t}(\omega)=0$ implies that there is a whole interval $[t,t+\epsilon_{\omega})$ where the right-continuous jump path $s \mapsto \mathbb{I}^M_{s}(\omega)$  is constantly zero. Similarly, we can show that the process $Y$ almost surely has left limits, which are of the form
\begin{align*}
Y_{t-}= \sum_{M \in  \mathcal{M}}  \mathbb{I}^M_{t-}  \frac{ \E_M [ X_{t-} \mathbb{I}^M_{t-} ]}{\E_M [ \mathbb{I}^M_{t-} ]}, \quad t > 0.
\end{align*}
According to Proposition \ref{DarstellungBedErwartBzglG}  $Y_{t-}$ almost surely equals
 $\E[ X_{t-} | \mathcal{G}^-_{t}]$.
As c\`{a}dl\`{a}g processes are uniquely defined by their values on separable subsets of the time line, our choice for $X^{\mathcal{G}}$ is almost surely the only possible modification of $(\E[X_t | \mathcal{G}_t])_{t\geq 0}$.

The  variation of $Y$ on $[0,t]$ is  bounded by
\begin{align*}
  &\sum_{M \in  \mathcal{M}_t} \sup_{\mathcal{T}}\sum_{\mathcal{T}} \bigg|   \mathbb{I}^M_{t_{k+1}}  \frac{ \E_M [ X_{t_{k+1}} \mathbb{I}^M_{t_{k+1}} ]}{\E_M [ \mathbb{I}^M_{t_{k+1}} ]} -\mathbb{I}^M_{t_{k}}  \frac{ \E_M [ X_{t_{k}} \mathbb{I}^M_{t_{k}} ]}{\E_M [ \mathbb{I}^M_{t_{k}} ]} \bigg|\\
  &\leq \sum_{M \in  \mathcal{M}_t} \sup_{\mathcal{T}}\sum_{\mathcal{T}} \bigg(
    \bigg| \frac{ \mathbb{I}^M_{t_{k+1}}}{\E_M [ \mathbb{I}^M_{t_{k+1}} ]}-\frac{ \mathbb{I}^M_{t_{k}}}{\E_M [ \mathbb{I}^M_{t_{k}} ]}\bigg|   \E_M [ |X_{t_{k+1}}| \mathbb{I}^M_{t_{k+1}}] \\
  & \qquad \qquad \qquad\qquad +     \frac{ \mathbb{I}^M_{t_{k}}}{\E_M [ \mathbb{I}^M_{t_{k}} ]}\E_M [ |X_{t_{k+1}} \mathbb{I}^M_{t_{k+1}}- X_{t_{k}} \mathbb{I}^M_{t_{k}}| ] \bigg),
\end{align*}
where $\mathcal{T}$ is any partition $0=t_0 < \cdots < t_n=t$ of the interval $[0,t]$.
Since  $C_M(\omega):= \sup_t \mathbb{I}^M_{t}(\omega)/\E_M [ \mathbb{I}^M_{t} ](\omega)$ is finite for almost each $\omega \in \Omega$, see Proposition \ref{LemmaBoundedFractions},  and the variation of $L_M(s):= E_M[\mathbb{I}^M_{s}] $  is bounded by $2$, the latter bound is dominated by
\begin{align*}
 &  \sum_{M \in  \mathcal{M}_t} \bigg(\Big( 2 C_M + \int_{[0,t]} \mathbb{I}^M_{s} \frac{1}{L_M(s)L_M(s-)}  \d |L_M|(s) \Big) \E_M \Big[\sup_{0\leq s \leq t} |X_s|\Big] \\
 & \qquad \qquad + C_M \E_M \Big[2\, \sup_{0\leq s \leq t} |X_s| + \int_{[0,t]} \mathbb{I}^M_{s} \d |X|_s \Big] \bigg)\\
 & \leq  \sum_{M \in  \mathcal{M}_t} \bigg( \big( 2 C_M +  2\,t \, C^2_M \big)  \E_M \Big[ \int_{[0,t]} \d |X|_s \Big] + 3 \, C_M \E_M \Big[\int_{[0,t]} \d |X|_s \Big]\bigg),
\end{align*}
which is finite for almost each $\omega \in \Omega$, since $X$ has integrable variation on compacts and since $\mathcal{M}_t(\omega)$ is finite.
\end{proof}

\section{Infinitesimal compensators}
In this section we derive infinitesimal compensators for a large class of incrementally adapted jump processes, including the counting processes  $t \mapsto \mu_I([0,t]\times B)$ for any $I \in \mathcal{N}$ and $B\in \mathcal{E}_I$.  Under the conventions $0/0:=0$  and \eqref{ConventionCondProb} let
\begin{align*}
  &\nu_I ([0,t] \times B )
   :=  \sum_{M \in \mathcal{M}} \int_{(0,t] \times B}  \mathbb{I}^M_{u-}  \frac{ P_{M,R_{I}=(u,e)}(A_{u-}^{M})}{P_M(A_{u-}^{M}) }   P_M^{R_I}(\d (u,e) ), \\
    & \rho_I([0,t] \times B ) :=  \sum_{M \in \mathcal{M}} \int_{(0,t] \times B}  \mathbb{I}^M_{u}  \frac{ P_{M,R_{I}=(u,e)}(A_{u}^{M})}{P_M(A_{u}^{M}) }   P_M^{R_I}(\d (u,e) )
\end{align*}
for $ t \geq 0$, $B \in \mathcal{E}_I$, $I \in \mathcal{N}$.
%
\begin{proposition}\label{Lemma2Finite}
For each $I  \in \mathbb{N}$ the  mappings $\nu_I$ and $\rho_I$  can be uniquely extended to random measures on $([0, \infty) \times E_I, \mathcal{B}([0, \infty) \times  E_I))$.
\end{proposition}
The proof of the proposition is given below.

{\color{black}
In the following we use the short notation
\begin{align*}
F.\kappa\,((0,t]\times B) := \int_{(0,t ]\times B} F(u,e) \,\kappa (\d (u,e))
\end{align*}
for random measures $\kappa$ and integrable random functions $F$.
\begin{theorem}\label{GeneralCompOfJumpProc}
Suppose that the  mappings $(t,e,\omega) \mapsto F_I(t,e)(\omega)$, $I \in \mathcal{N}$, are jointly measurable and satisfy
\begin{align}\label{IntegrCondJumpProcess}
\E\bigg[ \int_{(0,t ]\times E_I }  |F_I(u,e)| \, \mu_I (\d (u,e)) \bigg] < \infty.
\end{align}
 If $F_I(t,e)$ is $\mathcal{G}_t^-$-measurable for each $(t,e)$, then for each  $ B \in \mathcal{E}_I$ the
 jump process
\begin{align*}
&t \mapsto  F_I.\mu_I((0,t]\times B)
\end{align*}
is an IF-martingale with IF-compensator
\begin{align*}
&t \mapsto   F_I.\nu_I((0,t]\times B).
\end{align*}
 If $F_I(t,e)$ is $\mathcal{G}_t$-measurable for each $(t,e)$, then for each  $ B \in \mathcal{E}_I$ the
 jump process
\begin{align*}
&t \mapsto  F_I.\mu_I((0,t]\times B)
\end{align*}
is an IB-martingale with IB-compensator
\begin{align*}
&t \mapsto   F_I.\rho_I((0,t]\times B).
\end{align*}
\end{theorem}
}

By choosing $G_I\equiv 1$ and $G_{I'}\equiv 0$ for $I'\neq I$, Theorem \ref{GeneralCompOfJumpProc} yields in particular that $\nu_I$ is the IF-compensator and $\rho_I$ is the IB-compensator of the counting process $\mu_I$.  In intuitive notation we write this fact as
\begin{align*}
  \E[ \mu_I( \d t \times B) | \mathcal{G}^-_{t}] &= \nu_I( \d t \times B),\\
  \E[ \mu_I( \d t \times B) | \mathcal{G}_{t}] &= \rho_I( \d t \times B), \quad  B \in \mathcal{E}_I .
\end{align*}

The proofs of Proposition \ref{Lemma2Finite} and Theorem \ref{GeneralCompOfJumpProc} follow now in several steps.
\begin{lemma}\label{FiniteDIstrOfR}
For each $M \in \mathcal{M}$ and $t \geq 0$  we almost surely have
\begin{align}\begin{split}\label{Finite1A}
 \sum_{I \in \mathcal{N}} \int_{[0,t]\times E_I} P_M^{R_I}(\d (u,e) ) < \infty.
\end{split}\end{align}
\end{lemma}
\begin{proof}
For each $s \geq 0$ and $M \in \mathcal{M}$ assumption \eqref{FinExpJumpNo} implies that
\begin{align}\label{finiteExpect}
\E_M \bigg[ \sum_{j=1}^{\infty}  \mathbf{1}_{\{T_j \leq s\}} \bigg]  < \infty
\end{align}
almost surely.  Therefore, by applying the  Monotone Convergence Theorem we obtain
\begin{align}\begin{split}\label{Finite1A}
 \infty > \E_M \bigg[ \sum_{i=1}^{\infty}  \mathbf{1}_{\{T_i \leq s\}} \bigg] 
 &\geq \E_M \bigg[ \sum_{I \in \mathcal{N}}  \mu_I([0,s]\times E_I)  \bigg] \\
  & =   \sum_{I \in \mathcal{N}} \E_M \bigg[  \mu_I([0,s]\times E_I)  \bigg]\\
  &=  \sum_{I \in \mathcal{N}}  \int_{[0,s]\times E_I } P_M^{R_I}(\d (u,e) )
\end{split}\end{align}
almost surely for each $M \in \mathcal{M}$ and $s \geq 0$.
\end{proof}
\begin{proof}[Proof of Proposition \ref{Lemma2Finite}]
The processes $\mathbb{I}^M_{u-}$ and $P_M(A^M_{u-})$ are jointly measurable with respect to $(u,\omega)$, since $\mathbb{I}^M_{u-}$ and $P_M(A^M_{u-})$ are left-continuous in $u$, see Lemma \ref{LemmaLimits}. The mapping $P_{M,R_{I}=(u,e)}(A_{u-}^{M})$ is  jointly measurable with respect to $(u,e,\omega)$ since $P_{M,R_{I}=(u,e)}(A_{s-}^{M})$ is left-continuous in $s$ and jointly measurable with respect to $(u,e,\omega)\in [0,\infty)^{|I|}\times E_I\times \Omega$, see Lemma \ref{LemmaLimits}. Thus, for any fixed $A \in \mathcal{B}([0,\infty) \times E_I)$ the mapping $\omega \mapsto \nu_I ( A)(\omega)$ is measurable. Moreover, for almost each $\omega \in \Omega$ the mapping $A \mapsto \nu_I (A)(\omega)$  is a locally finite measure on $([0, \infty) \times E_I  , \mathcal{B}([0, \infty) \times E_I ))$.
This can be seen by combining Proposition \ref{LemmaBoundedFractions} and equation \eqref{Finite1A} and using the fact that $P_{M, R_{I}=(u,e)}(A_{u-}^{M})$ is bounded by $1$. Hence, $\nu_I$ has a unique extension to a
random measure on $([0, \infty) \times E_I  , \mathcal{B}([0, \infty) \times  E_I))$. Similar conclusions hold for the mappings $\rho_I$.
\end{proof}
{\color{black}
\begin{proposition}\label{CompensatorPerConstruction}
Suppose that the mappings $(t,e,\omega) \mapsto F_I(t,e)(\omega)$, $I \in \mathcal{N}$, are jointly measurable and satisfy \eqref{IntegrCondJumpProcess}. For each $t >0$ and $B  \in \mathcal{E}_I $  we almost surely have
\begin{align*}
&\lim_{n \rightarrow \infty} \sum_{\mathcal{T}_n} \E \big[F_I.\mu_I((t_k,t_{k+1}]\times B)  \big| \mathcal{G}_{t_k} \big]=G_I.\nu_I((0,t]\times B) ,\\
&\lim_{n \rightarrow \infty} \sum_{\mathcal{T}_n} \E \big[F_I.\mu_I((t_k,t_{k+1}]\times B)  \big| \mathcal{G}_{t_{k+1}} \big]=H_I.\rho_I((0,t]\times B)
\end{align*}
for any increasing sequence of partitions $(\mathcal{T}_n)_{n\in \mathbb{N}}$ of $[0,t]$ with $\lim_{n \rightarrow \infty} |\mathcal{T}_n|=0$ and for $G_I$ and $H_I$ defined by
\begin{align*}
G_I(u,e) &:= \sum_{M \in \mathcal{M}} \mathbb{I}^M_{u-}  \frac{\E_{M,R_I=(u,e)}[ \mathbb{I}^M_{u-} F_I(u,e) ]
}{\E_{M,R_I=(u,e)}[ \mathbb{I}^M_{u-} ]},\\
H_I(u,e) &:= \sum_{M \in \mathcal{M}} \mathbb{I}^M_{u}  \frac{\E_{M,R_I=(u,e)}[ \mathbb{I}^M_{u} F_I(u,e) ] }{\E_{M,R_I=(u,e)}[ \mathbb{I}^M_{u} ]}.
\end{align*}
\end{proposition}
}
\begin{proof}
By decomposing  $F$ into a positive part $F^+$ and a negative part $F^-$, it suffices to prove the first equation for the non-negative mappings $F^+$ and $F^-$ only.  Therefore, without loss of generality we suppose from now on that $F$ is non-negative.

Let $\mathcal{M}_t=\mathcal{M}_t(\omega)$ be defined as in \eqref{DefOfMt}. In the following we use the short notation $J_k:=(t_k,t_{k+1}]$.
Since $\sum_{M \in \mathcal{M}_t} \mathbb{I}^M_{t_k}=1$ for any $t_k$,  by applying \eqref{EquivOfCondExpectations}, the Monotone Convergence Theorem and the Law of Total Probability we obtain
\begin{align*}
& \E \big[ F_I.\mu_I(J_k \times B )  \big| \mathcal{G}_{t_k} \big]\\
 & = \sum_{M \in \mathcal{M}_t} \mathbb{I}^M_{t_k}
\frac{ \E_M [  \mathbb{I}^M_{t_k}  F_I.\mu_I(J_k \times B )   ]}{\E_M [ \mathbb{I}^M_{t_k} ]} \\
& = \sum_{M \in \mathcal{M}_t}  \int_{J_k\times E_I } \mathbb{I}^M_{t_k} \frac{\E_{M,R_I=(u,e) } [ \mathbb{I}^M_{t_k} F_I.\mu_I(J_k \times B )   ]  }{P_M(A_{t_k}^{M}) }  P^{R_{I}}_M( \d (u,e) )
\end{align*}
for  almost each $\omega \in \Omega$.
For $u \in (0,t]$ let $J^u$ be the unique interval $(t_k,t_{k+1}]$ from $\mathcal{T}_n$ such that $t_k< u \leq t_{k+1}$, and let $t(u)$ be the left end point of $J^u$. Then we can write
\begin{align*}
&\sum_{\mathcal{T}_n} \E \big[ F_I.\mu_I(J_k \times B)  \big| \mathcal{G}_{t_k} \big]\\
 &= \sum_{M \in \mathcal{M}_t} \int_{(0,t] \times E_I} \mathbb{I}_{t(u)}^{M}
 \frac{ \E_{M,R_{I}=(u,e)} [\mathbb{I}_{t(u)}^{M} F.\mu_I(J^u \times B  )    ]  }{P_M(A_{t(u)}^{M}) }  P_M^{R_{I}}(\d (u,e)).
\end{align*}
Taking the limit for $n\rightarrow \infty$, for almost each $\omega \in \Omega$ we obtain
\begin{align}\label{ZwischenGl1}\begin{split}
&\lim_{n \rightarrow \infty} \sum_{\mathcal{T}_n} \E \big[ F_I.\mu_I(J_k \times B)  \big| \mathcal{G}_{t_k} \big]\\
 &= \sum_{M \in \mathcal{M}_t}  \int_{(0,t] \times E_I} \lim_{n \rightarrow \infty}  \mathbb{I}_{t(u)}^{M}
 \frac{ \E_{M,R_{I}=(u,e)} [\mathbb{I}_{t(u)}^{M} F_I.\mu_I(J^u \times B )    ]  }{P_M(A_{t(u)}^{M}) }  P_M^{R_{I}}(\d (u,e)),
\end{split}\end{align}
 using that $\mathcal{M}_t$ is finite for almost each $\omega$ and applying the Monotone Convergence Theorem and the Dominated Convergence Theorem. Note that  Proposition \ref{LemmaBoundedFractions}, assumption \eqref{IntegrCondJumpProcess} and
 $ 0 \leq \mathbb{I}_{t(u)}^{M} F_I.\mu_I(J^u \times B )  \leq  F_I.\mu_I((0,t] \times B )$ ensure the existence of an integrable majorant.
For $n \rightarrow \infty$ we have $t(u) \uparrow u$ and $J^u \downarrow \{u\}$, so the Dominated Convergence Theorem implies that
\begin{align*}
 &\lim_{n \rightarrow \infty} \E_{M,R_{I}=(u,e)} [ \mathbb{I}_{t(u)}^{M} F_I.\mu_I(J^u \times B )     ]\\
 &= \E_{M,R_{I}=(u,e)} [ \mathbb{I}_{u-}^{M} \mathbf{1}_B(e) F_I(u,e) \mu_I(\{u\} \times \{e\} ) ]\\
 &= \mathbf{1}_B(e) \E_{M,R_{I}=(u,e)} [ \mathbb{I}_{u-}^{M}  F_I(u,e)  ].
\end{align*}
All in all we get that the right hand side of
equation \eqref{ZwischenGl1} equals $G_I.\nu_I((0,t ]\times B)$, and we can conclude that the first equation in Proposition \ref{CompensatorPerConstruction} holds. The proof of the second equation in Proposition \ref{CompensatorPerConstruction} is similar.
\end{proof}
\begin{proposition}\label{CompensatorIsItsOwnCompensator} Under the assumptions of {\color{black}Proposition \ref{CompensatorPerConstruction}}, for each $t \geq 0$  and $B \in \mathcal{E}_I $  we almost surely have
\begin{align*}
\lim_{n \rightarrow \infty} \sum_{\mathcal{T}_n} \E \big[ G_I.\nu_I((t_k,t_{k+1}] \times B)  \big| \mathcal{G}_{t_k} \big]& = G_I.\nu_I((0,t ]\times B),\\
\lim_{n \rightarrow \infty} \sum_{\mathcal{T}_n} \E \big[ H_I.\rho_I((t_k,t_{k+1}] \times B)  \big| \mathcal{G}_{t_{k+1}} \big] &= H_I.\rho_I((0,t ]\times B)
\end{align*}
for any increasing sequence of partitions $(\mathcal{T}_n)_{n\in \mathbb{N}}$ of $[0,t]$ with $\lim_{n \rightarrow \infty} |\mathcal{T}_n|=0$.
\end{proposition}

\begin{proof}
By decomposing  $G$ into a positive part $G^+$ and a negative part $G^-$, it suffices to prove the first equation for the non-negative mappings $G^+$ and $G^-$ only.  Therefore, without loss of generality we suppose from now on that $G$ is non-negative.

From the definition of $\nu_I$ and the Monotone Convergence Theorem we get
\begin{align*}
& \E[ G_I.\nu_I((0,t]\times E_I ) ] \\
& = \sum_{M \in \mathcal{M}} \E\bigg[  \E_{M}\bigg[  \int_{(0,t] \times E_I} G_I(u,e) \mathbb{I}^M_{u-}  \frac{ P_{M,R_{I}=(u,e)}(A_{u-}^{M})}{P_M(A_{u-}^{M}) }   P^{R_I}_M(\d (u,e) ) \bigg]  \bigg].
\end{align*}
From Proposition \ref{DarstellungBedErwartBzglG} we know that $G_I(u,e)$ is $\mathcal{G}^-_{u}$-adapted for each $(u,e)$. This fact and  \eqref{SpurSigmaAlgebren} imply that
\begin{align}\label{Ghieinziehen}
  G_I(u,e) \mathbb{I}^M_{u-} P_{M,R_{I}=(u,e)}(A_{u-}^{M}) =  \mathbb{I}^M_{u-} \E_{M,R_{I}=(u,e)}[\mathbb{I}^M_{u-} G_I(u,e)] .
\end{align}
By applying the Fubini-Tonelli Theorem and  the Monotone Convergence Theorem,  we obtain
\begin{align*}
&\E[ G_I.\nu_I((0,t]\times E_I ) ]\\
 &=\sum_{M \in \mathcal{M}} \E\bigg[   \int_{(0,t] \times E_I}  \E_{M}[ \mathbb{I}^M_{u-} ]  \frac{ \E_{M,R_{I}=(u,e)}[\mathbb{I}_{\tau_I-}^{M} G_I(Q_I,Z_I)]}{P_{M}(A_{u-}^{M}) }   P_{M}^{R_I}(\d (u,e) ) \bigg]\\
 &=\sum_{M \in \mathcal{M}}\E\Big[   \E_{M}\big[   \mathbb{I}^M_{\tau_I-} G_I(Q_I,Z_I) \mu_I((0,t]\times E_I)\big] \Big]\\
  &= \E\bigg[   G_I.\mu_I((0,t] \times E_I ) \bigg].
\end{align*}
The latter expectation is finite according to assumption \eqref{IntegrCondJumpProcess}. Hence, for each $M \in \mathcal{M}$ we almost surely have
\begin{align}\begin{split}\label{CompTheoEq2}
   &\E_M[ G_I.\nu_I((0,t]\times E_I )  ] < \infty,\\
  &G_I.\nu_I((0,t]\times E_I )  < \infty.
\end{split}\end{align}
Let $J_k:=(t_k,t_{k+1}]$. From the Dominated Convergence Theorem we obtain
\begin{align*}
   \lim_{n \rightarrow \infty} \sum_{\mathcal{T}_n}  \mathbb{I}^M_{t_k}  G_I.\nu_I(J_k \times B )  =  (\mathbb{I}^M_{\cdot -}G_I).\nu_I((0,t] \times B )
\end{align*}
since $\mathbb{I}^M$ is bounded by $1$ and because of the second line in \eqref{CompTheoEq2}. By using the first line in \eqref{CompTheoEq2}, the Dominated Convergence Theorem moreover yields
\begin{align*}
   &\lim_{n \rightarrow \infty} \sum_{\mathcal{T}_n}  \E_M[\mathbb{I}^M_{t_k} G_I.\nu_I(J_k \times B )]     =  \E_M\big[(\mathbb{I}^M_{\cdot -}G_I).\nu((0,t] \times B ) \big] .
\end{align*}
By applying the Fubini-Tonelli Theorem we can show that the latter equation equals
\begin{align*}
        \int_{(0,t] \times B}   \E_M[\mathbb{I}^M_{u-} G_I(u,e)]    \frac{P_{M, R_{I}=(u,e)}(A_{u-}^{M})}{P_{M}(A_{u-}^{M})}  P_M^{R_I}(\d (u,e) ).
\end{align*}
Using Proposition \ref{LemmaBoundedFractions} and the Dominated Convergence Theorem, we therefore obtain
\begin{align*}
& \lim_{n \rightarrow \infty} \sum_{\mathcal{T}_n} \mathbb{I}^M_{t_k} \frac{ \E_M [  \mathbb{I}^M_{t_k}  G_I.\nu_I(J_k \times B)  ]}{\E_M [ \mathbb{I}^M_{t_k} ]} \\
&=  \int_{(0,t] \times B} \frac{ \mathbb{I}^M_{u-} }{\E[\mathbb{I}^M_{u-}]}  \E_M[\mathbb{I}^M_{u-} G_I(u,e)]    \frac{P_{M, R_{I}=(u,e)}(A_{u-}^{M})}{P_{M}(A_{u-}^{M})}  P_M^{R_I}(\d (u,e) )\\
& = \int_{(0,t] \times B}  \mathbb{I}^M_{u-} G_I(u,e)     \frac{P_{M, R_{I}=(u,e)}(A_{u-}^{M})}{P_{M}(A_{u-}^{M})}  P_M^{R_I}(\d (u,e) ),
\end{align*}
where the second equality bases on  the fact that \eqref{SpurSigmaAlgebren} and the $\mathcal{G}^-_{u}$-adaptedness of $G_I(u,e,)$ allows us to pull $G_I(u,e)$ out of the conditional expectation $\E_M[\mathbb{I}^M_{u-} G_I(u,e)] $. Summing the latter equation over $M \in \mathcal{M}_t$ for $\mathcal{M}_t$ defined as in  \eqref{DefOfMt} and applying
Proposition \ref{DarstellungBedErwartBzglG}, we obtain
 \begin{align*}
\lim_{n \rightarrow \infty} \sum_{\mathcal{T}_n} \E \big[ G_I.\nu_I((t_k,t_{k+1}] \times B)  \big| \mathcal{G}_{t_k} \big] = G_I.\nu_I((0,t ]\times B)
\end{align*}
almost surely. Thus, we can conclude that the first equation in Proposition \ref{CompensatorIsItsOwnCompensator} holds. The proof of the second equation in Proposition \ref{CompensatorIsItsOwnCompensator} is similar.
\end{proof}
{\color{black}
\begin{proof}[Proof of Theorem \ref{GeneralCompOfJumpProc}]
Let $G_I$ and $H_I$ be defined as in Proposition \ref{CompensatorPerConstruction}.
If $F_I(t,e)$ is $\mathcal{G}_t^-$-measurable for each $(t,e)$, then Proposition \ref{DarstellungBedErwartBzglG} implies that $G_I(t,e)=F_I(t,e)$ almost surely. Similarly, if $F_I(t,e)$ is $\mathcal{G}_t$-measurable for each $(t,e)$, we have $H_I(t,e)=F_I(t,e)$ almost surely.
With this fact and  by subtracting the limit equations in Proposition \ref{CompensatorPerConstruction} and  Proposition \ref{CompensatorIsItsOwnCompensator}, we obtain that $G_I.\mu_I ([0,t] \times B)-G_I.\nu_I ([0,t] \times B)$ and $H_I.\mu_I ([0,t] \times B)-H_I.\rho_I ([0,t] \times B)$ satisfy the defining limit equations for IF/IB-martingales. The IF/IB-predictability of the compensators follows from  Proposition \ref{CompensatorIsItsOwnCompensator}.
Note that all involved processes are incrementally adapted to $\mathcal{G}$ because of
  \eqref{SpurSigmaAlgebren1} and \eqref{SpurSigmaAlgebren}.
\end{proof}
}

\section{Infinitesimal martingale representations}
Suppose that $\lambda_I$ is the compensator of $\mu_I$ with respect to $\mathcal{F}$.
For each  integrable random variable $\xi$, the classical martingale representation theorem yields that the martingale $X_t:=\E[ \xi | \mathcal{F}_t ]$ can be represented as
\begin{align}\label{ClassMartRepresOFCondExp}
  X_t =  X_0 + \sum_{I \in \mathcal{N}} \int_{(0,t]\times E_I} F_I(u,e) \Big(\mu_I(\d (u,e)) - \lambda_I(\d (u,e))\Big),
\end{align}
where $F(u,e)(\omega)$ is jointly measurable in $(u,e,\omega)$ and $\omega \mapsto F(u,e)(\omega)$ is $\mathcal{F}_{u-}$-measurable for each $(u,e)$, see e.g.~Karr (1986). We now extend this result to the non-monotone information $\mathcal{G}$.
\begin{theorem}\label{MartingaleLikeRepresentation}
Let  $\xi$ be an  integrable  random variable.
Then for each $t \geq 0$ equation \eqref{MartingRepres} 
holds almost surely  for
 \begin{align}\label{MartingRepres2}
G_I(s, u,e) &:= \sum_{M \in \mathcal{M}} \mathbb{I}^M_{s} \bigg( \frac{\E_{M,R_I=(u,e)}[ \mathbb{I}^M_{s} \xi ] }{\E_{M,R_I=(u,e)}[ \mathbb{I}^M_{s} ]} - \frac{\E_M[ \mathbb{I}^M_{u-}\mathbb{I}^M_{u} \xi ] }{\E_M[ \mathbb{I}^M_{u-}\mathbb{I}^M_{u}]} \bigg).
\end{align}
For each $I \in \mathcal{I}$ and $e \in E_I$ the process $u \mapsto G_I(u-,u,e)$ is $\mathcal{G}^-$-adapted and the process $u \mapsto G_I(u,u,e)$ is $\mathcal{G}$-adapted.
\end{theorem}
%
%
If the mappings $F_I(u,e)=G_I(u-,u,E)$ and $F_I(u,e)=G_I(u,u,e)$ satisfy the integrability condition in Theorem \ref{GeneralCompOfJumpProc}, then representation \eqref{MartingRepres} is a sum of IF-martingales and IB-martingales with respect to $\mathcal{G}$. In case of $\mathcal{F}= \mathcal{G}$ we have $\nu_I = \lambda_I, \rho_I =\mu_I$ and \eqref{MartingRepres} equals \eqref{ClassMartRepresOFCondExp}, so \eqref{MartingRepres} is a generalization of \eqref{ClassMartRepresOFCondExp}.

The proof of Theorem \ref{MartingaleLikeRepresentation} is given below. Recall that our notation uses the convention \eqref{ConventionCondProb}.
\begin{lemma}\label{LemmaForMRT}
Let   $\xi$ be an integrable  random variable. Then for each $t \geq 0$ we  have
\begin{align}\begin{split}\label{MartRepThepEq2}
&\E_M[ \mathbb{I}_{t}^{M} \xi ] -\E_M[ \mathbb{I}^M_{0} \xi  ]
  =   \sum_{I \in \mathcal{N}} \int_{(0,t] \times E_I}  \E_{M,R_{I}=(u,e)}\big[(\mathbb{I}^M_{u}-\mathbb{I}^M_{u-}) \xi\big]   P^{R_I}_M(\d (u,e) ).
\end{split}
\end{align}
\end{lemma}
\begin{proof}
As  \eqref{MartRepThepEq2} is additive in $\xi$, it suffices to show the equation for nonnegative and bounded random variables $\xi$ only. The general case follows then from the Monotone Convergence Theorem applied on the sequence $\xi_n := (\xi \wedge n) - (-\xi \wedge n)$, $n \in \mathbb{N}$. Therefore, in the remaining proof we suppose  that $0 \leq \xi \leq C $ for a finite real number $C$.

Let $U_{t_k}(\omega):= \sup\{ t \in (t_k ,\infty) : T_j(\omega) \not\in (t_k,t), j \in \mathbb{N}\}$, i.e.~$U_{t_k}$ gives the time of the first occurrence of  a random time strictly after $t_k$. Since 
$1=\sum_{I \in \mathcal{N}} \mathbf{1}_{\{U_{t_k}=Q_I\}} $ we can conclude that
\begin{align}\begin{split}\label{MartRepThepEq1}
  & \E_M[ \mathbb{I}_{t}^{M} \xi ] -\E_M[ \mathbb{I}^M_{0} \xi  ] \\
    &=  \lim_{n \rightarrow \infty } \sum_{\mathcal{T}_n} \sum_{I \in \mathcal{N}} \E_M\big[ \mathbf{1}_{\{U_{t_k}=Q_I\}}(\mathbb{I}^M_{t_{k+1}}-\mathbb{I}^M_{t_k}) \xi \big]  \\
  & = \lim_{n \rightarrow \infty }  \sum_{I \in \mathcal{N}}  \sum_{\mathcal{T}_n} \int_{(t_k,t_{k+1}] \times E_I}  \E_{M, R_{I}=(u,e)}\big[\mathbf{1}_{\{U_{t_k}=Q_I\}}(\mathbb{I}^M_{t_{k+1}}-\mathbb{I}^M_{t_k}) \xi\big]   P^{R_I}_M(\d (u,e) ),
\end{split}\end{align}
where we use the fact that  $ \mathbf{1}_{\{U_{t_k}=Q_I\}}(\mathbb{I}^M_{t_{k+1}}-\mathbb{I}^M_{t_k}) = 0$  unless $t_k< Q_I \leq t_{k+1}$.
 Because of \eqref{Finite1A} and
 $| \E_{M, R_{I}=(u,e)}\big[\mathbf{1}_{\{U_{t_k}=Q_I\}}(\mathbb{I}^M_{t_{k+1}}-\mathbb{I}^M_{t_k}) \xi\big] | \leq 2 C$, we can apply the Dominated Convergence Theorem on the last line in \eqref{MartRepThepEq1}, which leads to  \eqref{MartRepThepEq2}.
Note here that
$$\mathbf{1}_{\{U_{t_k}=Q_I=u\}}(\mathbb{I}^M_{t_{k+1}}-\mathbb{I}^M_{t_k}) \rightarrow \mathbf{1}_{\{Q_I=u\}} (\mathbb{I}_{u}^{M}-\mathbb{I}^M_{u-}) $$
 for $t_{k+1} \downarrow u$ and $t_k \uparrow u$ implies that
\begin{align*}
&\E_{M,R_{I}=(u,e)}\big[\mathbf{1}_{\{U_{t_k}=Q_I\}}(\mathbb{I}^M_{t_{k+1}}-\mathbb{I}^M_{t_k}) \xi \big] \rightarrow \E_{M,R_{I}=(u,e)} \big[(\mathbb{I}_{u}^{M}- \mathbb{I}^M_{u-}) \xi\big].
\end{align*}
\end{proof}
\begin{proof}[Proof of Theorem \ref{MartingaleLikeRepresentation}] Let $M \in \mathcal{M}$ be arbitrary but fixed, and define $M+1:= \{i+1 : i \in M\}$. If $\mathbb{I}^M_{u-}=1$,  then only random times from the index set
\begin{align*}
  M':=(\{1,3,\ldots\} \setminus M) \cup (M+1)
\end{align*}
can occur at time $u$. If $\mathbb{I}^M_{u}=1$,  then only random times from the index set
\begin{align*}
  M'':=M \cup (\{2,4,\ldots \} \setminus (M+1))
\end{align*}
can be equal  to $u$. Therefore, equation \eqref{MartRepThepEq2} can be represented as
\begin{align*}
\E_M[ \mathbb{I}_{t}^{M} \xi ] = K_t + L_t, \quad t\geq 0
\end{align*}
where
\begin{align*}
K_t :=& \sum_{ I \subseteq M'} \int_{(0,t] \times E_I}  \E_{M,R_{I}=(u,e)}\big[(\mathbb{I}^M_{u}-\mathbb{I}^M_{u-}) \xi\big]   P^{R_I}_M(\d (u,e) ) +\E_M[ \mathbb{I}_{0}^{M} \xi ]  \\
=&- \sum_{I \subseteq M'} \int_{(0,t] \times E_I}  \E_{M,R_{I}=(u,e)}\big[\mathbb{I}^M_{u-} \xi\big]   P^{R_I}_M(\d (u,e) ) +\E_M[ \mathbb{I}_{0}^{M} \xi ] , \\
L_t :=& \sum_{I \subseteq M''} \int_{(0,t] \times E_I}  \E_{M,R_{I}=(u,e)}\big[(\mathbb{I}^M_{u}-\mathbb{I}^M_{u-}) \xi\big]   P^{R_I}_M(\d (u,e) )\\
=& \sum_{I \subseteq M''} \int_{(0,t] \times E_I}  \E_{M,R_{I}=(u,e)}\big[\mathbb{I}^M_{u} \xi\big]   P^{R_I}_M(\d (u,e) ).
\end{align*}
Furthermore, using $M'\cap M'' = \emptyset$, we can show that
\begin{align*}
   \E_M[ \mathbb{I}_{t}^{M}  \xi ] -\E_M[ \mathbb{I}_{t-}^{M}  \mathbb{I}_{t}^{M} \xi ]
  &=   \sum_{  I \subseteq M''} \E_M[ \mathbb{I}_{t}^{M}  \mathbf{1}_{\{Q_I=t\}} \xi ]  \\
  & =   \sum_{  I \subseteq M''} \int_{\{t\} \times E_I}    \E_{M,R_{I}=(u,e)}\big[\mathbb{I}^M_{u} \xi\big]   P^{R_I}_M(\d (u,e))\\
  & =  \Delta L_t
\end{align*}
for $t>0$, which implies that $\E_M[ \mathbb{I}_{t-}^{M}  \mathbb{I}_{t}^{M} \xi ]= K_{t}+L_{t-}$. Analogously we obtain
\begin{align*}
   \E_M[ \mathbb{I}_{t-}^{M}  \xi ] -\E_M[ \mathbb{I}_{t-}^{M}  \mathbb{I}_{t}^{M} \xi ]
  = - \Delta K_t.
\end{align*}
In the specific case $\xi=1$ we write $k_t$ and $l_t$ instead of $K_t$ and $L_t$.
By applying integration by parts path-wise for each $\omega \in \Omega$,  we get that
\begin{align*}
   & ( K_{t} +L_{t-}) \,\d \mathbb{I}_{t}^{M}  +  \mathbb{I}_{t-}^{M} \d K_t  + \mathbb{I}_{t}^{M} \d L_t\\
   &= \d \Big( \mathbb{I}_{t}^{M}  \E_M[ \mathbb{I}_{t}^{M}  \xi ]\Big)\\
      &= \d \bigg( \frac{\mathbb{I}_{t}^{M}  \E_M[ \mathbb{I}_{t}^{M}  \xi ]}{\E_M[ \mathbb{I}_{t}^{M}  ]}\E_M[ \mathbb{I}_{t}^{M}   ]\bigg)\\
& = (k_{t} + l_{t-})\, \d  \bigg( \frac{\mathbb{I}_{t}^{M} \E_M[ \mathbb{I}_{t}^{M}  \xi ]}{\E_M[ \mathbb{I}_{t}^{M}   ]}\bigg) +   \frac{\mathbb{I}_{t-}^{M} \E_M[ \mathbb{I}_{t-}^{M}  \xi ] }{\E_M[ \mathbb{I}_{t-}^{M}  ]} \d k_t+  \frac{\mathbb{I}_{t}^{M} \E_M[ \mathbb{I}_{t}^{M}  \xi ] }{\E_M[ \mathbb{I}_{t}^{M}   ]} \d l_t\\
& = (k_{t} + l_{t-})\, \d  \bigg( \frac{\mathbb{I}_{t}^{M} \E_M[ \mathbb{I}_{t}^{M}  \xi ]}{\E_M[ \mathbb{I}_{t}^{M}   ]}\bigg) +   \mathbb{I}_{t-}^{M} \frac{ K_{t-}+L_{t-}}{k_{t-}+l_{t-}} \d k_t+  \mathbb{I}_{t}^{M} \frac{ K_{t}+L_{t}}{k_{t}+l_{t}} \d l_t.
\end{align*}
The equation that is formed by the first and last line can be rewritten to
\begin{align*}
   & ( K_{t} +L_{t-}) \,\d \mathbb{I}_{t}^{M}  + \frac{k_{t} +l_{t-}}{k_{t-} +l_{t-}} \mathbb{I}_{t-}^{M} \d K_t  + \frac{k_{t} +l_{t-}}{k_{t} +l_{t}} \mathbb{I}_{t}^{M} \d L_t\\
& = (k_{t} + l_{t-})\, \d  \bigg( \frac{\mathbb{I}_{t}^{M} \E_M[ \mathbb{I}_{t}^{M}  \xi ]}{\E_M[ \mathbb{I}_{t}^{M}   ]}\bigg)   +   \mathbb{I}_{t-}^{M} \frac{ K_{t}+L_{t-}}{k_{t-}+l_{t-}} \d k_t+  \mathbb{I}_{t}^{M} \frac{ K_{t}+L_{t-}}{k_{t}+l_{t}} \d l_t,
\end{align*}
since $\Delta K_t$  and $\Delta L_t$ are dominated by $\Delta k_t$ and $\Delta l_t$ and because of
\begin{align*}
  &\bigg(\frac{k_{t} +l_{t-}}{k_{t-} +l_{t-}} -1 \bigg) \mathbb{I}_{t-}^{M} \Delta K_t +  \bigg(\frac{k_{t} +l_{t-}}{k_{t} +l_{t}} -1 \bigg) \mathbb{I}_{t}^{M} \Delta L_t\\
  &=\mathbb{I}_{t-}^{M} \bigg( \frac{ K_{t}+L_{t-}}{k_{t-}+l_{t-}}-\frac{ K_{t-}+L_{t-}}{k_{t-}+l_{t-}}\bigg) \Delta k_t+  \mathbb{I}_{t}^{M} \bigg( \frac{ K_{t}+L_{t-}}{k_{t}+l_{t}}-\frac{ K_{t}+L_{t-}}{k_{t}+l_{t}}\bigg) \Delta k_t.
\end{align*}
Under the convention $0/0:=0$ and by using the Radon-Nikodym Theorem we may
multiply $(k_t+l_{t-})^{-1}$ on both hand sides, which leads to
\begin{align}\label{PartIntEquat3}\begin{split}
   &  \frac{\E_M[ \mathbb{I}_{t-}^{M} \mathbb{I}_{t}^{M}  \xi ] }{\E_M[ \mathbb{I}_{t-}^{M}  \mathbb{I}_{t}^{M}   ]}   \,\d \mathbb{I}^M_t +   \frac{\mathbb{I}_{t-}^{M}  }{\E_M[ \mathbb{I}_{t-}^{M}   ]}\d K_t +  \frac{\mathbb{I}_{t}^{M}  }{\E_M[ \mathbb{I}_{t}^{M}   ]}\d L_t \\
& =  \d  \bigg( \frac{\mathbb{I}_{t}^{M} \E_M[ \mathbb{I}_{t}^{M}  \xi  ]}{\E_M[ \mathbb{I}_{t}^{M}    ]}\bigg)
 +   \frac{\E_M[ \mathbb{I}_{t-}^{M} \mathbb{I}_{t}^{M}  \xi ] }{\E_M[ \mathbb{I}_{t-}^{M}  \mathbb{I}_{t}^{M}   ]} \bigg( \frac{\mathbb{I}_{t-}^{M}  }{\E_M[ \mathbb{I}_{t-}^{M}   ]}\d k_t +  \frac{\mathbb{I}_{t}^{M}  }{\E_M[ \mathbb{I}_{t}^{M}  ]}\d l_t \bigg).
\end{split}\end{align}
Because of \eqref{MartRepThepEq2} and $\d \mathbb{I}_{t}^{M}  =  \sum_{I \in \mathcal{N}} (\mathbb{I}_{t}^{M}  -\mathbb{I}_{t-}^{M} ) \mu_I (\d t\times  E_I )$, the latter equation can be rewritten to
\begin{align}\begin{split}   \label{MRTmainequation}
  &   \sum_{I \in \mathcal{N}} \frac{ \E_M[ \mathbb{I}_{t-}^{M} \mathbb{I}_{t}^{M}  \xi  ] }{
     \E_M[ \mathbb{I}_{t-}^{M} \mathbb{I}_{t}^{M}   ]}    (\mathbb{I}_{t}^{M}  -\mathbb{I}_{t-}^{M} ) \mu_I(\d t \times E_I )  \\
     &  -  \sum_{I \in \mathcal{N}} \int_{E_I}  \mathbb{I}_{t-}^{M}  \frac{\E_{M,R_I=(t,e)}[ \mathbb{I}_{t-}^{M}   \xi  ] }{\E_{M,R_I=(t,e)}[ \mathbb{I}_{t-}^{M}    ]} \nu_I (\d t \times \d e ) \\
     & \quad +  \sum_{I \in \mathcal{N}} \int_{E_I}  \mathbb{I}_{t}^{M}  \frac{\E_{M,R_I=(t,e)}[ \mathbb{I}_{t}^{M}   \xi  ] }{\E_{M,R_I=(t,e)}[ \mathbb{I}_{t}^{M}   ]} \rho_I (\d t \times \d e ) \\
       &= \d  \bigg( \frac{\mathbb{I}_{t}^{M} \E_M[ \mathbb{I}_{t}^{M}  \xi ]}{\E_M[ \mathbb{I}_{t}^{M}   ]}\bigg)\\
&\quad +   \sum_{I \in \mathcal{N}}\frac{ \E_M[ \mathbb{I}_{t-}^{M} \mathbb{I}_{t}^{M}  \xi ] }{   \E_M[ \mathbb{I}_{t-}^{M} \mathbb{I}_{t}^{M}    ]}  \Big(- \mathbb{I}_{t-}^{M}  \nu_I( \d t \times E_I ) + \mathbb{I}_{t}^{M}  \rho_I( \d t \times E_I )    \Big).
\end{split}\end{align}
Let $I \in \mathcal{N}$ be arbitrary but fixed. Then for each  $M \in \mathcal{M}$  there exists an $\tilde{M} \in \mathcal{M}$, and vice versa,  such that
\begin{align*}
   \mathbb{I}_{t-}^{M} \mu_I( \d t \times \d e )=  \mathbb{I}_{t}^{\tilde{M}} \mu_I( \d t \times \d e ).
\end{align*}
As a consequence, for almost each $\omega \in \Omega $  we  have
\begin{align}\begin{split}\label{MRTtheoremNullEq}
0 = \sum_{M \in \mathcal{M}} \sum_{I \in \mathcal{N}}  \int_{E_I} \bigg( &\mathbb{I}_{t}^{M}  \frac{\E_{M,R_I=(t,e)}[ \mathbb{I}_{t}^{M}   \xi  ] }{\E_{M,R_I=(t,e)}[ \mathbb{I}_{t}^{M}    ]} -  \mathbb{I}_{t-}^{M}  \frac{\E_{M,R_I=(t,e)}[ \mathbb{I}_{t-}^{M}   \xi  ] }{\E_{M,R_j=(t,e)}[ \mathbb{I}_{t-}^{M}   ]} \bigg) \,  \mu_I( \d t \times \d e ).
\end{split}\end{align}
Because of
\begin{align*}
\d \E[ \xi | \mathcal{G}_t] & = \sum_{M \in \mathcal{M}} \d  \bigg( \frac{\mathbb{I}_{t}^{M} \E_M[ \mathbb{I}_{t}^{M}  \xi  ]}{\E_M[ \mathbb{I}_{t}^{M}  ]}\bigg),
\end{align*}
by summing  equation \eqref{MRTmainequation} over $M \in \mathcal{M}$ and adding  \eqref{MRTtheoremNullEq}, for almost each $\omega \in \Omega$ we end up with \eqref{MartingRepres} and \eqref{MartingRepres2} after rearranging the addends.
By applying Proposition \ref{DarstellungBedErwartBzglG} we can see that $G_I(u-,u,e)$ is $\mathcal{G}_{u-}$-measurable and that $G_I(u,u,e)$ is $\mathcal{G}_u$-measurable for each $(I,e)$.
\end{proof}

\section{Infinitesimal representations for optional projections}
Suppose that $X$ is a c\`{a}dl\`{a}g process that satisfies \eqref{IntegrabCond} and such that $X_t-X_0$ is  $\mathcal{F}_t$-measurable for each $t \geq 0$. Then the optional projection of $X$ with respect to $\mathcal{F}$ can be represented as
\begin{align}\label{ExtendedClassMartRepr} \begin{split}
  &\d \E[ X_t | \mathcal{F}_t] =     \d X_t +  \sum_{I \in \mathcal{N}}\int_{E_I } F_I(t,e)  \Big( \mu_I(\d t \times \d e)) -  \lambda_I(\d t \times \d e  ) \Big)
\end{split}\end{align}
for  mappings $F_I(t,e)$ that are $\mathcal{F}$-predictable processes in argument $t$ for each $(I,e)$. In order to see that, apply the classical martingale representation theorem on the $\mathcal{F}$-martingale
\begin{align*}
  \E[ X_0 | \mathcal{F}_t]-\E[ X_0 | \mathcal{F}_0]=\E[ X_t | \mathcal{F}_t]-\E[ X_0 | \mathcal{F}_0] - ( X_t -X_0)
\end{align*}
and rearrange the addends.
The following theorem extends \eqref{ExtendedClassMartRepr} to non-monotone information settings.
\begin{theorem}\label{MartingaleLikeRepresentation2}
Let $X$ be a c\`{a}dl\`{a}g process that satisfies \eqref{IntegrabCond} and that has an IB-compensator with respect to $\mathcal{G}$, denoted as $X^{IB}$. Then
\begin{align}\label{MartingRepres3}\begin{split}
 \E[ X_t | \mathcal{G}_t]-\E[ X_0 | \mathcal{G}_0]
 &=  X^{IB}_t + \sum_{I \in \mathcal{N}} \int_{(0,t ]\times E_I } G_I(u-,u,e) \,(\mu_I-\nu_I) (\d (u,e)) \\
 &\quad \qquad + \sum_{I \in \mathcal{N}} \int_{(0,t]\times E_I} G_I(u,u,e) \,(\rho_I- \mu_I) (\d (u,e))
\end{split}\end{align}
almost surely with
\begin{align}\label{MartingRepres4}\begin{split}
G_I(s, u,e) &:= \sum_{M \in \mathcal{M}} \mathbb{I}^M_{s} \bigg( \frac{\E_{M,R_I=(u,e)}[ \mathbb{I}^M_{s} X_{u-} ] }{\E_{M,R_I=(u,e)}[ \mathbb{I}^M_{s} ]} - \frac{\E_M[ \mathbb{I}^M_{u-}\mathbb{I}^M_{u} X_{u-} ] }{\E_M[ \mathbb{I}^M_{u-}\mathbb{I}^M_{u} ]} \bigg).
\end{split}\end{align}
If $X$ has an IF-compensator with respect to $\mathcal{G}$, denoted as $X^{IF}$, then
\eqref{MartingRepres3} still holds but with  $X_t^{IB}$ replaced by $X_t^{IF}$ and $X_{u-}$ replaced by $X_u$ in \eqref{MartingRepres4}.
\end{theorem}
By applying Proposition \ref{DarstellungBedErwartBzglG} we can see that $G_I(u-,u,e)$ is $\mathcal{G}^-_{u}$-measurable and that $G_I(u,u,e)$ is $\mathcal{G}_u$-measurable. Hence,  the integrals in the first and  second line of \eqref{MartingRepres3} describe  IF-martingales and  IB-martingales with respect to $\mathcal{G}$ if $F_I(u,e)=G_I(u-,u,e)$ and $F_I=G_I(u,u,e)$ satisfy the integrability condition \eqref{IntegrCondJumpProcess}, see the comments below Theorem \ref{GeneralCompOfJumpProc}.

In the special case $\mathcal{G}= \mathcal{F}$ we have $\nu_I = \lambda_I$,  $\rho_I =\mu_I$,  $X=X^{IB}$ and the representations  \eqref{MartingRepres3} and  \eqref{MartingaleLikeRepresentation2} are equivalent, i.e.~\eqref{MartingRepres3} is a generalization of  \eqref{MartingaleLikeRepresentation2}.

Even if $\mathcal{G} \subsetneqq \mathcal{F}$ we can still have $X=X^{IB}$ or $X=X^{IB}$. The following  example presents   non-trivial processes $X$ that equal their IB-compensators or their IF-compensators.
\begin{example}\label{ExpSojournPaym}
Let  $h(M,t)(\omega): \mathcal{M} \times [0,\infty) \times \Omega \rightarrow \mathbb{R}$ be measurable and let $|h(M,t)| \leq Z$ for an integrable majorant $Z$.  Let
 $\gamma$ be the sum of the Lebesgue measure and a countable number of Dirac measures,
\begin{align*}
  \gamma (B )=  \lambda (B) + \sum_{i=1}^{\infty} \delta_{t_i}(B), \quad B \in \mathcal{B}([0,\infty)),
\end{align*}
for deterministic times points $0\leq t_1 < t_2 < ... $ that are increasing to infinity.
Then the c\`{a}dl\`{a}g process $X$ defined by
\begin{align*}
  X_t:= \sum_{M \in \mathcal{M}} \int_{[0,t]}  \mathbb{I}_{s}^{M}  h(M,s)\,\gamma (\d s)
\end{align*}
has the IB-compensator
\begin{align*}
 X^{IB}_t= \int_{(0,t]}\sum_{M \in \mathcal{M}} \mathbb{I}_{s}^{M}\E[ h(M,s)| \mathcal{G}_s] \,  \gamma(\d s).
\end{align*}
In order to see that, apply Proposition \ref{DarstellungBedErwartBzglG},  the Dominated Convergence Theorem, Proposition \ref{LemmaBoundedFractions}  and Lemma \ref{LemmaLimits} in order to obtain that
\begin{align*}
    &\lim_{n \rightarrow \infty} \sum_{\mathcal{T}_n} \E[X_{t_{k+1}}- X_{t_{k}} |\mathcal{G}_{t_{k+1}}]\\
    &=  \lim_{n \rightarrow \infty} \sum_{\mathcal{T}_n} \int_{(t_k,t_{k+1}]}  \sum_{M \in \mathcal{M}_t} \mathbb{I}^M_{t_{k+1}} \E\bigg[ \sum_{\tilde{M} \in \mathcal{M}}  h(\tilde{M},s) \mathbb{I}^{\tilde{M}}_{s} \bigg|\mathcal{G}_{t_{k+1}}\bigg] \gamma (\d s)\\
     &= \sum_{M \in \mathcal{M}_t} \lim_{n \rightarrow \infty} \sum_{\mathcal{T}_n} \int_{(t_k,t_{k+1}]}\mathbb{I}^M_{t_{k+1}} \frac{\E_M\Big[   \sum_{\tilde{M} \in \mathcal{M}} \mathbb{I}^M_{t_{k+1}} h(\tilde{M},s)   \mathbb{I}^{\tilde{M}}_{s} \Big]}{\E_M[ \mathbb{I}^{M}_{t_{k+1}} ]} \gamma (\d s)\\
     &= \sum_{M \in \mathcal{M}_t}  \int_{(0,t]}\mathbb{I}^M_{s} \frac{\E_M[    h(M,s)   \mathbb{I}^M_{s} ]}{\E_M[ \mathbb{I}^M_{s} ]} \gamma (\d s)\\
     & =   \sum_{M \in \mathcal{M}} \int_{(0,t]}\mathbb{I}^M_{s}  \E[ h(M,s)| \mathcal{G}_s]   \, \gamma(\d s)
 \end{align*}
 almost surely, where $\mathcal{M}_t$ is defined as in \eqref{DefOfMt}. In case that $s\mapsto h(M,s)$ is $\mathcal{G}$-adapted for each $M$ we have $X=X^{IB}$.
  Likewise we can show that
 the c\`{a}dl\`{a}g process
 \begin{align*}
  Y_t:= \sum_{M \in \mathcal{M}} \int_{[0,t]}  \mathbb{I}_{s-}^{M}  h(M,s)\,\gamma (\d s)
\end{align*}
has the IF-compensator
\begin{align*}
 Y^{IF}_t= \int_{(0,t]}\sum_{M \in \mathcal{M}} \mathbb{I}_{s-}^{M}\E[ h(M,s)| \mathcal{G}_{s-}] \,  \gamma(\d s).
\end{align*}
In case that $s\mapsto h(M,s)$ is $\mathcal{G}^-$-adapted for each $M$   we  have $Y=Y^{IF}$.
\end{example}
\begin{proof}[Proof of Theorem \ref{MartingaleLikeRepresentation2}]
The theorem follows from the additive decomposition
\begin{align*}
&\E[ X_t | \mathcal{G}_t]-\E[ X_0 | \mathcal{G}_0]\\
 &= \lim_{n \rightarrow \infty} \sum_{\mathcal{T}_n} (\E[ X_{t_{k+1}} | \mathcal{G}_{t_{k+1}}] - \E[ X_{t_k} | \mathcal{G}_{t_k}])\\
  & = \lim_{n \rightarrow \infty} \sum_{\mathcal{T}_n}   \E[ X_{t_{k+1}}-X_{t_{k}} | \mathcal{G}_{t_{k+1}}]  +  \lim_{n \rightarrow \infty} \sum_{\mathcal{T}_n} (\E[ X_{t_k} | \mathcal{G}_{t_{k+1}}] - \E[ X_{t_k} | \mathcal{G}_{t_{k}}])
\end{align*}
and from applying Theorem \ref{MartingaleLikeRepresentation}  for each addend $\E[ X_{t_k} | \mathcal{G}_{t_{k+1}}] - \E[ X_{t_k} | \mathcal{G}_{t_{k}}]$. The sum $\sum_{\mathcal{T}_n} (\E[ X_{t_k} | \mathcal{G}_{t_{k+1}}] - \E[ X_{t_k} | \mathcal{G}_{t_{k}}])$ has a representation of the form \eqref{MartingRepres} in case of $t_k < s \leq t_{k+1}$ for $G_I(s, u,e)$ defined by \eqref{MartingRepres4}.
Because of the c\`{a}dl\`{a}g property of $X$, by applying the Dominated Convergence Theorem pathwise for almost each $\omega \in \Omega$, we end up with \eqref{MartingRepres3} and \eqref{MartingRepres4}.
 The alternative decomposition
\begin{align*}
&\E[ X_t | \mathcal{G}_t]-\E[ X_0 | \mathcal{G}_0]\\
 & = \lim_{n \rightarrow \infty} \sum_{\mathcal{T}_n}   \E[ X_{t_{k+1}}-X_{t_{k}} | \mathcal{G}_{t_{k}}]  +  \lim_{n \rightarrow \infty} \sum_{\mathcal{T}_n} (\E[ X_{t_{k+1}} | \mathcal{G}_{t_{k+1}}] - \E[ X_{t_{k+1}} | \mathcal{G}_{t_{k}}])
\end{align*}
leads to the second variant where $X^{IB}$ is replaced by $X^{IF}$ and $X_{u-}$ is replaced by $X_u$ in \eqref{MartingRepres4}.
\end{proof}
\begin{remark}\label{RemarkIntuitRepresG}
Without loss of generality suppose here that $0 \not\in E$.
Motivated by Remark \ref{RemarkInformRepres}, for any $t > 0$ and any integrable random variable $\xi$ let
\begin{align*}
  \E[ \xi|\mathcal{G}_{t}, R_I=(t,e) ]&:=  \E[ \xi |(\Gamma_t, R_I)=\cdot \,](\Gamma_t,(t,e)), \\
  \E[ \xi |\mathcal{G}_{t-}, R_I=(t,e) ]&:=  \E[ \xi |(\Gamma_{t-}, R_I)=\cdot \,](\Gamma_{t-},(t,e)) , \quad e \in E_I,\\
  \E[ \xi |\mathcal{G}_{t}, \Delta_t=m ]&:=  \E[ \xi |(\Gamma_t,  \Delta_t)=\cdot \,](\Gamma_t,m) ,\\
   \E[ \xi |\mathcal{G}_{t-}, \Delta_t=m ]&:=  \E[ \xi |(\Gamma_{t-},  \Delta_t)=\cdot \,](\Gamma_{t-},m) ,\quad m \in \{0,1\},
\end{align*}
where the random variable $\Delta_t:= \sum_{I \in \mathcal{N}} \mu_I(\{t\}\times E_I)$ indicates whether there is a stopping event at time $t$.
One can then show that the integrands in \eqref{MartingRepres3} almost surely equal
\begin{align*}
G_I(t-, t,e) &= \E[  X_{t-} |\mathcal{G}_{t-}, R_I=(t,e) ] - \E[ X_{t-} |\mathcal{G}_{t-}, \Delta_t=0 ],\\
G_I(t, t,e) &= \E[  X_{t-} |\mathcal{G}_{t}, R_I=(t,e) ] - \E[  X_{t-} |\mathcal{G}_{t}, \Delta_t=0 ],
\end{align*}
for each $t>0$, $I \in \mathcal{N}$, and $e \in E_I$. The differences on the right hand side have  intuitive interpretations: The  first line describes the difference in expectation between a change scenario  and a remain scenario if we are currently at time $t-$ and are looking forwards in time. Similarly, the second line describes the difference in expectation between a change scenario  and a remain scenario if we are currently at time $t$ and are looking backwards in time. In formula \eqref{MartingRepres3} these differences in expectation are integrated with respect to the compensated forward and backward scenario dynamics.
\end{remark}

\section{Examples}

Here we come back to Examples \ref{IntExpLifeIns} and \ref{IntExpCredRat} and show how our infinitesimal martingale representations can be applied in life insurance and credit risk modeling.

\begin{example}[Evaluation of life insurance based on big data]\label{ExpThiele1}
Consider a life  insurance contract where the insurer collects health-related information about the insured with the aim to improve forecasts of the individual future insurance liabilities. For example, this can involve data from activity trackers or social media. Here, the marked point process includes the time of death $\tau_1$, which is recorded as $\zeta_1:=\tau_1$, and further health-related information $(\tau_i, \zeta_i)_{i \geq 2}$.
Upon request of the policyholder with reference to the 'right to erasure' according to the General Data Protection Regulation of the European Union, or as a self-imposed data privacy effort of the data provider, the insurer deletes parts of the health related data at certain time points, i.e.~we expand $(\tau_i, \zeta_i)_{i \geq 2}$ by  deletion times $(\sigma_i)_{i \geq 2}$. For completeness we define $\sigma_1:= \infty$.

In the classical insurance modelling without data deletion, the time dynamics of the expected future insurance payments is commonly described by Thiele's equation, see e.g.~M{\o}ller (1993) and Djehiche \& L\"{o}fdahl (2016). Suppose that $A_t$ gives the aggregated benefit cash flow of the life insurance contract on $[0,t]$, including  survival benefits with rate $a(t)$ and a death benefit of $\alpha(t)$ upon death at time $t$, i.e.
\begin{align*}
  A_t = \int_{0}^t \mathbf{1}_{\{\tau_1 > s\}} a(s)  \d s +  \int_{[0,t]\times E} \alpha(s) \, \mu_{\{1\}}(\d(s,e)), \quad t \geq 0.
\end{align*}
We assume here that $a:[0,\infty) \rightarrow \mathbb{R}$ and $\alpha:[0,\infty) \rightarrow \mathbb{R}$ are bounded.
For a given  interest intensity $\phi:[0,\infty)  \rightarrow [0,\infty)$ and a finite contract horizon of $[0,T]$, the process
\begin{align*}
  X_t :=\int_{(t, T]} e^{-\int_t^s \phi(u) \d u} \d A_s
\end{align*}
describes the discounted future liabilities of the insurer seen from time $t$.
As the  c\`{a}dl\`{a}g  process $X=(X_t)_{t \geq 0}$ is neither adapted to $\mathcal{F}$ nor adapted to $\mathcal{G}$, an insurer has to work with the optional projection instead (the so-called prospective reserve), i.e.~the insurer aims to calculate
\begin{align*}
  X^{\mathcal{F}}_t = \E[ X_t | \mathcal{F}_t ], \quad t \geq 0
\end{align*}
 in case that there is no data deletion and
\begin{align*}
  X^{\mathcal{G}}_t = \E[ X_t | \mathcal{G}_t ], \quad t \geq 0
\end{align*}
in case that information deletions may occur. The process $X^{\mathcal{G}}$ is well defined c\`{a}dl\`{a}g process according to Theorem \ref{PropDefOptionalProjection}.

By applying \eqref{ExtendedClassMartRepr} and It\^{o}'s Lemma,  we can derive the so-called stochastic Thiele equation
\begin{align}\label{ThieleBSDE}
  \d X^{\mathcal{F}}_t & = - \d A_t  +  \phi(t)\,X^{\mathcal{F}}_{t-} \d t + \sum_I \int_{E_I } F_I(t,e)  \, (\mu_I -\lambda_I )( \d t \times \d e )
\end{align}
with terminal condition   $X^{\mathcal{F}}_T=0$, cf.~M{\o}ller (1993, equation (2.17)). The integrand $F_I(t,e)$, which almost surely equals
$$F_I(t,e) =\E[  X_{t-} |\mathcal{F}_{t-}, R_I=(t,e) ] - \E[ X_{t-} |\mathcal{F}_{t-}, \Delta_t=0 ]$$
according to Remark \ref{RemarkIntuitRepresG}),  is a key figure in life insurance risk management, known as the sum at risk.  Equation \eqref{ThieleBSDE}  can be interpreted as a backward stochastic differential equation (BSDE) with solution $(X^{\mathcal{F}},(F_I)_I)$, see Djehiche \& L\"{o}fdahl (2016) for Markovian models and Christiansen \& Djehiche (2020) for non-Markovian models. The BSDE \eqref{ThieleBSDE} is in particular relevant if the life insurance payments $a$ and $\alpha$ depend on the current policy value such that the insurance cash flow  $A$ is only implicitly defined.

By applying Theorem \ref{MartingaleLikeRepresentation2} and It\^{o}'s Lemma and using the fact that process $A$ equals its own IB-compensator (since $\sigma_1 = \infty$), we are able to derive an analogous equation for $X^{\mathcal{G}}$, namely
\begin{align}\label{ThieleWRTG}\begin{split}
\d X^{\mathcal{G}}_t  = - \d A_t  +  \phi(t)\,X^{\mathcal{G}}_{t-} \d t &+ \sum_I\int_{E_I } G_I(t-,t,e)  \, (\mu_I -\nu_I )( \d t \times \d e ) \\
 &+ \sum_I \int_{E_I } G_I(t,t,e)  \, (\rho_I -\mu_I )( \d t \times \d e )
\end{split}\end{align}
with terminal condition   $X^{\mathcal{G}}_T=0$. This equation can be interpreted as a new type of BSDE with solution $(X^{\mathcal{G}},(G_I)_I)$, featuring an IF-martingale and an IB-martingale instead of a classical martingale. 
 The IF-martingale in the first line describes the impact of new information on the optional projection $X^{\mathcal{G}}$.   The IB-martingale in the second line quantifies the effect of information deletions on $X^{\mathcal{G}}$. The integrands $G_I(t-,t,e)$ and $G_I(t,t,e)$, which almost surely equal
\begin{align*}
G_I(t-, t,e) &= \E[  X_{t-} |\mathcal{G}_{t-}, R_I=(t,e) ] - \E[ X_{t-} |\mathcal{G}_{t-}, \Delta_t=0 ],\\
G_I(t, t,e) &= \E[  X_{t-} |\mathcal{G}_{t}, R_I=(t,e) ] - \E[  X_{t-} |\mathcal{G}_{t}, \Delta_t=0 ]
\end{align*}
according to Remark \ref{RemarkIntuitRepresG}, generalize the classical definition of the sum at risk. They are needed in life insurance risk management for sensitivity analyses, safe-side calculations, contract modifications and surplus decompositions.

If the policyholder may decide about data deletions at discretion, then the resulting value changes of the insurance contract can be systematically exploited by the policyholder, leading to a kind of data privacy arbitrage. Since it is the IB-martingale  in \eqref{ThieleWRTG} that measures the value changes due to data deletions at times $(\sigma_i)_{i \geq 2}$, it represents the potential data privacy arbitrage. A simple solution for avoiding data privacy arbitrage could be to charge the IB-martingale as a fee upon a data deletion request.  The fee can also be negative and represents then a bonus payment. However, more complex risk sharing schemes will be needed in insurance practice that moreover distinguish between different causes for data deletions. By following the concept of Schilling et al.~(2020) to interpret martingale representations as risk factor decompositions, we may interpret the infinitesimal martingale parts in \eqref{ThieleWRTG} as an  additive surplus decomposition that can distinguish between numerous kinds of jump events $\mu_I$, $I\subset \mathbb{N}$. Such an additive decomposition of the insurer's surplus  is an important step for aligning  insurance risk management to the digital age.
\end{example}

\begin{example}[Markovian approximations in credit rating models]\label{ExpThiele2}
A popular approximation concept in credit rating modeling is to pretend that the credit rating process is Markovian even if the empirical data does not fully support this assumption.
Suppose that credit ratings are updated at integer times only. By setting $\tau_i:=i-1$ and $\sigma_i:=i$ for $i \in \mathbb{N}$ and defining $\zeta_i$ as the credit rating at time $\tau_i$, the rating process $R=(R_t)_{t \geq 0}$ has the representation
\begin{align*}
R_t= \sum_{i=1}^{\infty} \zeta_i \mathbf{1}_{\{\tau_i \leq t < \sigma_{i} \}}, \quad t \geq 0,
\end{align*}
and satisfies
$$\mathcal{G}_t=\sigma(R_t) \vee \mathcal{Z}, \quad t \geq 0. $$
The jumps of process $R$ correspond to the random counting measures $\mu_I$. In the Jarrow-Lando-Turnbull model the rating space $E$ is finite, $(R_i)_{i \in \mathbb{N}_0}$ is assumed to be a Markov-chain, and
\begin{align}\label{JLTmodelPQ}
Q(R_{i+1}= r_{i+1}| R_{i}=r_i ) = \pi(i,r_i)\, P(R_{i+1}= r_{i+1}| R_{i}=r_i ),  \quad r_i,r_{i+1} \in E,\, i \in \mathbb{N}_0,
\end{align}
where $Q$ is the equivalent martingale measure and $\pi$ is a deterministic function on $\mathbb{N}_0 \times E$. The latter formula allows us to estimate $Q$ from market data by a two step method: First, the transition probabilities $P(R_{i+1}= r_{i+1}| R_{i}=r_i )$ are estimated from observed credit rating time series. Then the function $\pi$ is calibrated such that the risk-neutral values of  credit rating derivatives conform with observed market prices.  Once we have $Q$, we can use
the (classical) martingale representation \eqref{ClassMartRepresOFCondExp} in order to explicitly construct hedges for financial claims $\xi$, see e.g.~Last \& Penrose (2011). For example, by arguing analogously to formula \eqref{ThieleBSDE},  the claim $\xi=h(R_T)$ has  the hedging strategy
\begin{align}\label{HedgingF}\begin{split}
h(R_T)  - B(0) \, \E_Q\bigg[ \frac{h(R_T)}{B(T)} \bigg| \mathcal{F}_0 \bigg] &= \int_{(0,T]} B(t-) \, \E_Q\bigg[ \frac{h(R_T)}{B(T)} \bigg| \mathcal{F}_{t-} \bigg]  \, \frac{B(\d t)}{B(t-)}\\ &\quad + \sum_{I \in \mathcal{N}} \int_{(0,T]\times E_I} F_I(t,e) \big(\mu_I- \lambda_I\big)(\d t \times \d e).
\end{split}\end{align}
The integral in the first line describes the investments in the risk-free asset $B$. The second line corresponds to risky investments and can be rewritten in terms of the tradable assets in a complete  financial market, cf.~Section 5 in Last \& Penrose (2011).

A standard estimator for the stage occupation probabilities of the Markov-chain $R$  w.r.t.~$P$ is the Aalen-Johansen estimator, which directly corresponds to the Nelson-Aalen estimator for the compensators $\lambda_I$ of the random counting measures $\mu_I$. Under the assumption that $R$ is Markovian, the Nelson-Aalen estimator consistently estimates $\lambda_I=\nu_I$. If $R$ is not Markovian, then the Nelson-Aalen estimator still consistently estimates $\nu_I$, see Datta \& Satten (2001), but now $\nu_I \neq \lambda_I$. In other words, if we ignore the information beyond $\mathcal{G}$ in the estimation of $\lambda_I$  due to an incorrect  Markov assumption, then we actually estimate the infinitesimal forward compensator $\nu_I$ instead of the classical compensator $\lambda_I$.
Similarly, ignoring the information beyond $\mathcal{G}$ upon estimating $F_I$ and  \eqref{EvaluationFormula} from market data means that we unintentionally end up with
\begin{align*}
  G_I(t-,t,e) &= B(0) \,\E_Q\bigg[ \frac{h(R_T)}{B(T)} \bigg|\mathcal{G}_{t-}, R_I=(t,e) \bigg] - B(0) \,\E_Q\bigg[ \frac{h(R_T)}{B(T)} \bigg|\mathcal{G}_{t-}, \Delta_t=0 \bigg]
\end{align*}
and  $B(0) \,\E_Q[ h(R_T)/B(T) |\mathcal{G}_{t}] $ rather than
\begin{align*}
  F_I(t,e) &=B(0) \,\E_Q\bigg[ \frac{h(R_T)}{B(T)} \bigg|\mathcal{F}_{t-}, R_I=(t,e) \bigg] - B(0) \,\E_Q\bigg[ \frac{h(R_T)}{B(T)} \bigg|\mathcal{F}_{t-}, \Delta_t=0 \bigg]
\end{align*}
and  $B(0) \,\E_Q[ h(R_T)/B(T) |\mathcal{F}_{t}] $.
For the correct interpretation of the latter conditional expectations with mixed conditions see Remark \ref{RemarkIntuitRepresG}. All in all, by ignoring the information beyond $\mathcal{G}$ in the estimation and calculation of the hedging strategy \eqref{HedgingF} due to  an incorrect  Markov assumption,  we unintentionally  end up with the trading strategy
\begin{align}\label{HedgingG}\begin{split}
  &\int_{(0,T]} B(t-)\,\E_Q\bigg[ \frac{h(R_T)}{B(T)} \bigg| \mathcal{G}_{t-} \bigg]  \,  \frac{B(\d t)}{B(t-)} \\
  &+ \sum_{I \in \mathcal{N}} \int_{(0,T]\times E_I} G_I(t-,t,e) \big(\mu_I- \nu_I\big)(\d t \times \d e)
\end{split}\end{align}
instead of the hedging strategy on the right hand side of \eqref{HedgingF}. Is \eqref{HedgingG} still a hedge for $h(R_T)$? By applying Theorem \ref{MartingaleLikeRepresentation} instead of \eqref{ClassMartRepresOFCondExp} and using that $\mathcal{F}_0=\mathcal{G}_0$ and $\mathcal{G}_T= \sigma(R_T) \vee \mathcal{Z}$, analogously to \eqref{HedgingF} we get
\begin{align}\label{HedgingGCorrection}\begin{split}
 h(R_T) - B(0) \, \E_Q\bigg[ \frac{h(R_T)}{B(T)} \bigg| \mathcal{F}_0 \bigg] &= \int_{(0,T]}B(t-)\, \E_Q\bigg[ \frac{h(R_T)}{B(T)} \bigg| \mathcal{G}_{t-} \bigg]  \, \frac{B(\d t)}{B(t-)}\\ &\quad + \sum_{I \in \mathcal{N}} \int_{(0,T]\times E_I} G_I(t-,t,e) \big(\mu_I- \nu_I\big)(\d t \times \d e)\\
 &\quad + \sum_{I \in \mathcal{N}} \int_{(0,T]\times E_I} G_I(t,t,e) \big(\rho_I- \mu_I\big)(\d t \times \d e).
\end{split}\end{align}
Equation \eqref{HedgingGCorrection} implies that  the trading strategy \eqref{HedgingG} is actually not a hedge for $\xi=h(R_T)$. The  hedging error is given by the third line in \eqref{HedgingGCorrection}. To sum it up, by estimating and calculating the hedging strategy \eqref{HedgingF} under an incorrect Markov assumption for $R$, we unintentionally replace the (classical) $\mathcal{F}$-martingale in \eqref{HedgingF} by the $\mathcal{G}$-IF-martingale in \eqref{HedgingG} (the risk-free investment is also affected), and the corresponding $\mathcal{G}$-IB-martingale is just the hedging error.

Schilling et al.~(2020) interpret martingale representations as additive risk factor decompositions.  Likewise we can read  the (infinitesimal) martingale parts in \eqref{HedgingF} and \eqref{HedgingGCorrection}  as  linear risk factor decompositions. The relevance of such decompositions in credit risk modelling is explained in Rosen \& Saunders (2010).
\end{example}

\section*{References}

\bigskip {\small
\begin{list}{}{\leftmargin1cm\itemindent-1cm\itemsep0cm}
%

\item {Bandini, E., 2015.  Existence and uniqueness for backward stochastic differential equations driven by a random measure, possibly non quasi-left continuous. Electron. Commun. Probab. 20/71.  doi:10.1214/ECP.v20-4348.}

\item {Boel, R., Varaiya, P., Wong, E., 1975. Martingales on jump processes I: representation
results. SIAM J. Control \& Optimization 13,999-1021.}


\item {Confortola, F., 2019. $ L^p $ solution of backward stochastic differential equations driven by a marked point process. Mathematics of Control, Signals, and Systems, 31(1), 1.}

\item {Cohen, S.N., 2013. A martingale representation theorem for a class of jump processes. arXiv:1310.6286v1.}

\item {Cohen, S.N., Elliott, R.J., 2008. Solutions of Backward Stochastic Differential Equations on Markov Chains. Commun. Stoch. Anal.. 2. 10.31390/cosa.2.2.05.}

\item {Cohen, S.N., Elliott, R.J., 2010. Comparisons for backward stochastic differential equations on Markov chains and related no-arbitrage conditions. The Annals of Applied Probability, 20(1), 267-311.}


\item {Chou, C.S., Meyer, P.-A., 1975. Sur la r\'{e}presentation des martingales comme
int\'{e}grales stochastiques dans les processus ponctuels. In: S\'{e}minaire de Probabilit
\'{e}s IX, Lecture Notes in Mathematics 465. Springer-Verlag.}

\item {Christiansen, M.C., Djehiche, B., 2020. Nonlinear reserving and multiple contract modifications in life insurance. Insurance: Mathematics and Economics 93, 187-195.}

\item {Datta, S., Satten, G.A., 2001. Validity of the Aalen–Johansen estimators of stage occupation
probabilities and Nelson–Aalen estimators of integrated
transition hazards for non-Markov models, Statistics \& Probability Letters 55, 403-411.}

\item {Davis, M.H.A., 1976. The representation of martingales of jump processes, SIAM J.
Control \& Optimization 14, 623-638.}


\item {Davis, M.H.A., 2005. Martingale representation and all that. In:  Advances in control, communication networks, and transportation systems, p.~57-68. Birkh\"{a}user Boston.}

\item {Delong, {\L}., 2013. Backward stochastic differential equations with jumps and their actuarial and financial applications. London: Springer.}

\item {Djehiche, B., L\"{o}fdahl, B., 2016. Nonlinear reserving in life insurance: Aggregation and mean-field approximation.
Insurance: Mathematics and Economics 69, 1-13.}

\item {Elliott, R.J., 1976. Stochastic integrals for martingales of a jump process with partially
accessible jump times.  Z.~Wahrscheinlichkeitstheorie ver.~Geb 36,213-266.}


\item {Jacod, J., 1975. Multivariate point processes: Predictable projections, Radon-Nikodym derivatives, representation of martingales. Z.~Wahrscheinlichkeitstheorie ver.~Geb.~31, 235-253.}

\item {Jarrow, R.A., Lando, D., Turnbull, S.M., 1997. A Markov model for the term structure of credit risk spreads. The review of financial studies, 10(2), 481-523.}

\item {Karr, A., 1986. Point processes and their statistical inference. Dekker, New York.}

\item{ Lando, D., Skodeberg, T., 2002. Analyzing rating transitions and rating drift with continuous
observations. Journal of Banking and Finance, 26,423–444.}

 \item {Last, G., Penrose, M.D., 2011. Martingale representation for Poisson processes with applications to minimal variance hedging. Stochastic Processes and their Applications 121(7), 1588-1606.}

\item {McNeil, A.J., Frey, R., Embrechts, P., 2015. Quantitative risk management: concepts, techniques and tools--revised edition. Princeton university press.}

\item {M{\o}ller, C.M., 1993. A stochastic version of Thiele's differential equation.
Scandinavian Actuarial Journal 1, 1-16.}

\item {M{\o}ller, T., Steffensen, M., 2007. Market-valuation methods in life and pension insurance. Cambridge University Press.}
\item {Norberg, R., 1992. Hattendorff's theorem and Thiele's differential equation generalized. Scandinavian Actuarial Journal, 1992(1), 2-14.}

\item {Norberg, R., 1999. A theory of bonus in life insurance. Finance and Stochastics, 3(4), 373-390.}

\item {Norberg, R., 2003. The Markov chain market. ASTIN Bulletin, 33(2), 265-287.}

\item {Norberg, R., 2005. Anomalous PDEs in Markov chains: domains of validity and numerical solutions. Finance and Stochastics, 9(4), 519-537.}

\item {Norberg, R., 2013. Optimal hedging of demographic risk in life insurance. Finance and Stochastics, 17(1), 197-222.}


\item {Pardoux, \'{E}., Peng, S., 1994. Backward doubly stochastic differential equations and systems of quasilinear SPDEs. Probability Theory and Related Fields, 98(2), 209-227.}

\item {Rosen, D., Saunders, D., 2010. Risk factor contributions in portfolio credit risk models. Journal of Banking \& Finance, 34, 336-349.}

\item {Schilling, K., Bauer, D., Christiansen, M.C., Kling, A., 2020. Decomposing Dynamic Risks into Risk Components.  Management Science 66(12), 5485-6064.}

\item {Tang, H.,  Wu, Z., 2013. Backward stochastic differential equations with Markov chains and related asymptotic properties. Advances in Difference Equations, 2013(1), 285.}
%
%
%
%
\end{list}}

\end{document}